\documentclass[12pt]{article}
\RequirePackage[OT1]{fontenc}
\RequirePackage{amsthm,amsmath,amssymb,subcaption,graphicx,epstopdf,enumerate,lmodern}
\RequirePackage[round,colon,authoryear]{natbib}
\RequirePackage{bigints}
\RequirePackage[colorlinks,citecolor=blue,urlcolor=blue]{hyperref}
\usepackage{bm,color,fancyvrb,xcolor,mathtools}
\usepackage{amscd}
\usepackage{mathrsfs}
\usepackage{bbm}

\def\eps{\varepsilon}
\font\tencmmib=cmmib10 \skewchar\tencmmib '60
\newfam\cmmibfam
\textfont\cmmibfam=\tencmmib

\def\lessim{\ \lower4pt\hbox{$
		\buildrel{\displaystyle <}\over\sim$}\ }
\def\gessim{\ \lower4pt\hbox{$\buildrel{\displaystyle >}
		\over\sim$}\ }

\usepackage{algorithm,algorithmicx,algpseudocode}
\usepackage{tensor}

\newtheorem{theorem}{Theorem}[section]
\newtheorem{proposition}[theorem]{Proposition}
\newtheorem{lemma}{Lemma}

\DeclarePairedDelimiter{\norm}{\lVert}{\rVert}
\DeclarePairedDelimiter{\abs}{\lvert}{\rvert}

\providecommand{\abs}[1]{\left\lvert#1\right\rvert}
\providecommand{\norm}[1]{\left\lVert#1\right\rVert}
\renewcommand{\hat}{\widehat}

\renewcommand{\hat}{\widehat}

\newcommand{\bfm}[1]{\ensuremath{\mathbf{#1}}}
\newcommand\numberthis{\addtocounter{equation}{1}\tag{\theequation}}

\def\ba{\bfm a}

   \def\bE{\bfm E}  \def\EE{\mathbb{E}}

   \def\bM{\bfm M}  
     \def\NN{\mathbb{N}}
     
     \def\PP{\mathbb{P}}
     
     \def\RR{\mathbb{R}}
   \def\bS{\bfm S}  \def\SS{\mathbb{S}}
   \def\bT{\bfm T}  
\def\bu{\bfm u}   \def\bU{\bfm U}  
\def\bv{\bfm v}   \def\bV{\bfm V}  
\def\bw{\bfm w}     
\def\bx{\bfm x}   \def\bX{\bfm X}  
   \def\bY{\bfm Y}  
   \def\bZ{\bfm Z}

\def\calL{{\cal  L}}

\def\calS{{\cal  S}}

\DeclareMathOperator{\argmin}{argmin}

\DeclareMathOperator{\diag}{diag}

\DeclareMathOperator{\Var}{Var}

\DeclareMathOperator{\tr}{tr}

\def\eps{\varepsilon}

\def\dsum{\displaystyle\sum}
\def\newpage{\vfill\eject}

\newcommand{\vertiii}[1]{{\left\vert\kern-0.25ex\left\vert\kern-0.25ex\left\vert #1 
		\right\vert\kern-0.25ex\right\vert\kern-0.25ex\right\vert}}

\usepackage{tikz}
\usetikzlibrary{decorations.pathreplacing,calc, patterns}

\def\scrE{\mathscr{E}}
\def\scrX{\mathscr{X}}
\def\scrT{\mathscr{T}}
\DeclareMathOperator{\argmax}{argmax}

\begin{document}
	
	\title{On Estimating Rank-One Spiked Tensors in the Presence of Heavy Tailed Errors$^\ast$}
	\author{Arnab Auddy and Ming Yuan\\
		Department of Statistics\\
		Columbia University}
	\date{(\today)}
	
	\maketitle

\begin{abstract}
In this paper, we study the estimation of a rank-one spiked tensor in the presence of heavy tailed noise. Our results highlight some of the fundamental similarities and differences in the tradeoff between statistical and computational efficiencies under heavy tailed and Gaussian noise. In particular, we show that, for $p$th order tensors, the tradeoff manifests in an identical fashion as the Gaussian case when the noise has finite $4(p-1)$th moment. The difference in signal strength requirements, with or without computational constraints, for us to estimate the singular vectors at the optimal rate, interestingly, narrows for noise with heavier tails and vanishes when the noise only has finite fourth moment. Moreover, if the noise has less than fourth moment, tensor SVD, perhaps the most natural approach, is suboptimal even though it is computationally intractable. Our analysis exploits a close connection between estimating the rank-one spikes and the spectral norm of a random tensor with iid entries. In particular, we show that the order of the spectral norm of a random tensor can be precisely characterized by the moment of its entries, generalizing classical results for random matrices. In addition to the theoretical guarantees, we propose estimation procedures for the heavy tailed regime, which are easy to implement and efficient to run. Numerical experiments are presented to demonstrate their practical merits.
\end{abstract}

	\footnotetext[1]{
		This research was supported by NSF Grant DMS-2015285.}

\newpage	
\section{Introduction}
Singular value decomposition (SVD) and principal component analysis (PCA) are among the most commonly used procedures in multivariate data analysis. See, e.g., \cite{tAND84a, pcabook}. By seeking low rank approximations to a data matrix, they allow us to reduce the dimensionality of the data, and oftentimes serve as a useful first step to capture the essential features in the data. While both were first developed for the analysis of data matrices, extensions to higher order tensors have also been developed in recent years. See, e.g., \cite{de2000multilinear, lu2008mpca,liu2017characterizing}. More generally, low rank tensor methods have exploded in popularity in numerous areas involving high dimensional data analysis. See \cite{kolda2009tensor,anandkumar2014tensor,cichocki2015tensor,sidiropoulos2017tensor} for recent reviews.

To fix ideas, consider a rank-one spiked tensor model
\begin{equation}
\label{eq:spike}
\mathscr{X}=\lambda \bu_1\otimes \bu_2\otimes\dots\otimes \bu_p+\mathscr{E},
\end{equation}
where the ``singular value'' $\lambda\ge 0$ is a scalar, and ``singular vectors'' $\bu_k$s are unit length vectors in $\RR^d$, and $\mathscr{E}\in \RR^{d\times \cdots\times d}$ is a noise tensor whose entries are independent and identically distributed random variables with zero mean and unit variance. The goal is to estimate the singular vectors after observing $\mathscr{X}$ in a high dimensional setting where $d$ is large. In particular, the special case when the noise tensor $\mathscr{E}$ consists of independent standard normal entries has attracted much attention in recent years, and an intriguing gap in statistical efficiencies with or without computational constraints is observed. It can be shown that tensor SVD that seeks the best rank-one approximation to $\mathscr{X}$ yields a consistent estimate of the singular vectors whenever $\lambda\gg d^{1/2}$. Hereafter, we say an estimate $\hat{\bu}_k$ of $\bu_k$ is consistent iff $\sin\angle(\hat{\bu}_k,\bu_k)\to 0$ as $d\to\infty$ where $\angle(\hat{\bu}_k,\bu_k)$ is the angle between two vectors $\hat{\bu}_k$ and $\bu_k$ taking value in $[0,\pi/2]$. However, computing the best rank-one approximation is known to be NP hard in general \citep[see, e.g.,][]{hackbusch2012tensor, hillar2013}. On the other hand, consistent yet computationally tractable estimates are only known when $\lambda \gtrsim d^{p/4}$. Hereafter $a\gtrsim b$ means that there is a constant $C$ independent of $d$ such that $a\ge Cb$. More specifically, it can be achieved by power iteration initialized with higher order SVD \citep[HOSVD; see, e.g.,][]{de2000best, de2000multilinear}. While a rigorous argument remains elusive, it is widely conjectured that $d^{p/4}$ is the tight algorithmic threshold below which no consistent estimates can be computed in polynomial time. It is instructive to consider the case when there are independent Gaussian errors, and the signal strength $\lambda\sim d^\xi$. These results can then be summarized by the following diagram. When $\xi>1/2$, the tensor SVD estimate $\hat{\bu}^{\rm SVD}_k$ is consistent, and indeed can be shown to be minimax rate optimal. Meanwhile, we only know of polynomial time computable estimators that are consistent if $\xi>p/4$. The shaded region between $\xi=1/2$ and $\xi=p/4$ in Figure \ref{fig:gauss} therefore signifies the tradeoff between statistical and computational efficiencies. 
\begin{figure}[htbp]
\begin{center}
\begin{tikzpicture}[scale=0.75]
\fill[black!30!white,fill opacity=0.8](0,-1.3) rectangle (2,0);
\path[-latex,draw,line width=2pt,black!80!white] (0,0) -- (8,0);
\path[draw,line width=2pt,black!80!white] (0,-1.5) -- (0,1.5);
\path[draw,line width=2pt,black!80!white] (2,-1.5) -- (2,0);
\path[draw,dashed,line width=2pt,black!80!white] (0,0) -- (-1.5,0);
\draw[thick,black](0,-1.8)node{$\xi={1 \over 2}$};
\draw[thick,black](2,-1.8)node{$\xi={p \over 4}$};
\draw[thick,black](8.2,0)node{\scriptsize $\xi$};
\draw[thick,black](5,-0.75)node{\small{$\sin\angle(\hat{\bu}_k,\bu_k)\sim d^{1/2-\xi}$}};
\draw[thick,black](3,0.75)node{\small{$\sin\angle(\hat{\bu}^{\rm SVD}_k,\bu_k)\sim d^{1/2-\xi}$}};
\draw[thick,black](10,0.75)node{Tensor SVD};
\draw[thick,black](12,-0.75)node{Computationally Tractable};
\draw[thick,black,align=center](-3,0)node{No consistent\\ estimator};
\draw[thick,black](1,-0.75)node{\scalebox{0.8}{\textcolor{blue!60!white}{}}};
\end{tikzpicture}
\end{center}
\caption{Tradeoff between statistical and computational efficiencies in estimating spiked rank-one tensors under Gaussian noise.}\label{fig:gauss}
\end{figure}
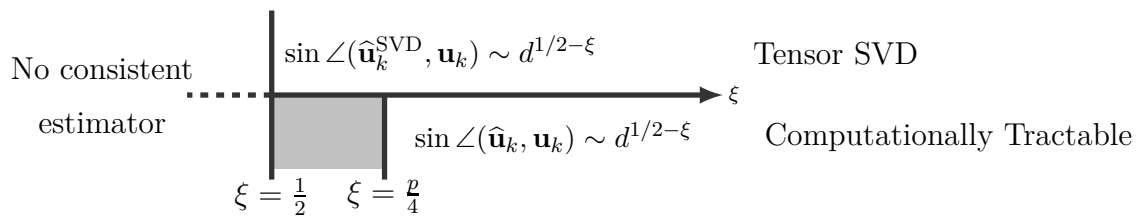
See, e.g., \cite{richard2014statistical, hopkins2015tensor, hopkins2016fast, liu2017characterizing, arous2019landscape} among many others. These observations can also be generalized beyond rank-one signals. See, e.g., \cite{zhang2018tensor, auddy2020perturbation}.

The Gaussian, or more generally subgaussian, assumption on the noise tensor $\scrE$, however, could be too restrictive in practice and neglecting departures from such assumptions could lead to erroneous results. For example, \cite{eklund2016cluster} showed how using Gaussian model based methods lead to very high false positive rate in fMRI studies. \cite{purdom2005error} and \cite{ringberg2007sensitivity} observed similar phenomena in genomic studies and anomaly detection respectively. Unfortunately, very little is known about the fundamental limit for estimating the rank-one spikes and the effect of computational constraints in the presence of heavy-tailed noise. A notable exception is the recent work of \cite{ding2020estimating} who developed polynomial time algorithms to recover the singular vectors $\bu_k$s through self avoiding walks and random coloring. They assume that the singular vectors are randomly sampled and therefore provide an average case analysis of their algorithms. More specifically, for third order tensor ($p=3$), if the entries of the error tensor has finite second moment, then their algorithm produces weak recovery when $\lambda\gtrsim d^{3/4}$. Moreover, their algorithm yields consistent estimates of the singular vectors if higher order moment conditions, e.g., finite 12th moment, are satisfied. Our work is inspired by this earlier development and aims at developing more practical algorithms for estimating spiked rank-one tensors and precise characterization of how the tradeoff between computational and statistical efficiency manifests beyond subgaussian errors. More specifically, we show that there are polynomial time computable estimates of $\bu_k$ that are not only consistent but also rate optimal whenever $\lambda\gg d^{p/4}\cdot{\rm polylog}(d)$ where ${\rm polylog}(d)$ is a certain polynomial of $\log d$.

The most natural approach to, and a useful benchmark for, estimating $\bu_k$s is the tensor SVD. Denote by $\hat{\bu}_k^{\rm SVD}$s the tensor SVD estimates of $\bu_k$s. We prove that if the entries of $\scrE$ have finite $\alpha$th moment for some $\alpha>4$, then with high probability,
\begin{equation}
\label{eq:optrate}
\max_{1\le k\le p}\sin\angle(\hat{\bu}_k^{\rm SVD},\bu_k)=O_p\left({\sqrt{d}\over \lambda}\right),
\end{equation}
as $d\to \infty$, provided that
$$
\lambda\gtrsim \left[d^{1/2}(\log d)^{1/2}+d^{(p-1)/\alpha+1/4}(\log d)^{3/2}\right].
$$
The above requirement on the signal-to-noise ratio can also be shown to be optimal, up to the logarithmic factor. More specifically, if the entries of $\scrE$ do not have finite $\alpha$th moment, then
$$
\sin\angle(\hat{\bu}_k^{\rm SVD},\bu_k)\to_p 1,
$$
for any
$$
\lambda\lesssim \left(d^{1/2}+d^{(p-1)/\alpha+1/4}\right).
$$
It is worth noting that the bounds on $\lambda$ highlights the intuitive facts that, under the same moment condition, estimating $\bu_k$s tends to be harder for higher order tensors, e.g., larger $p$; and for tensors of the same order, estimating $\bu_k$s tends to be easier with higher order moment, e.g., larger $\alpha$.

It is, however, well known that the tensor SVD is computationally infeasible in general. A common strategy to alleviate the computational expenses of the tensor SVD is through power iteration with spectral initialization. The rationale behind this is the presumptive optimality of the tensor SVD. A good initialization may ensure the resulting estimate, computable in polynomial time, inherits such optimality. We show that this is indeed the case: if $\lambda\gtrsim d^{p/4}$, then this yields a polynomial time computable estimate $\hat{\bu}_k$ such that
$$
\max_{1\le k\le p}\sin\angle(\hat{\bu}_k,\bu_k)=O_p\left({\sqrt{d}\over \lambda}\right).
$$
The signal strength requirement for polynomial time computable methods matches that under Gaussian noise and is strictly stronger than that for the tensor SVD estimate. Therefore, the tradeoff between computational and statistical efficiency remains. In particular, if we consider the case when $\lambda\sim d^\xi$, then our observations can be summarized by the diagram of Figure \ref{fig:a4}. The gap between the signal-to-noise ratio requirement for tensor SVD and polynomial computable estimators is the same as in the Gaussian case when $\alpha\ge 4(p-1)$ but narrows as $\alpha$ decreases to 4.
\input{norm4up.tex}

A more intriguing phenomenon occurs when the entries of $\scrE$ only has finite $\alpha$th moment for some $2<\alpha<4$. In this situation, we prove that \eqref{eq:optrate} holds if
$$
\lambda\gtrsim d^{{p-1\over \alpha}+{1\over 2}}(\log d)^{3/2}
$$
and the tensor SVD estimate $\hat{\bu}_k^{\rm SVD}$ is asymptotically perpendicular to $\bu_k$ of
$$\lambda\lesssim d^{p/\alpha}.$$
This can be summarized by the diagram of Figure \ref{fig:a2}. Interestingly, the tensor SVD is actually suboptimal in this case and there is an alternative estimator that is both computationally tractable and can attain the optimal rate of convergence whenever
$$
\lambda\gtrsim d^{p/4}(\log d)^{1/4}.
$$
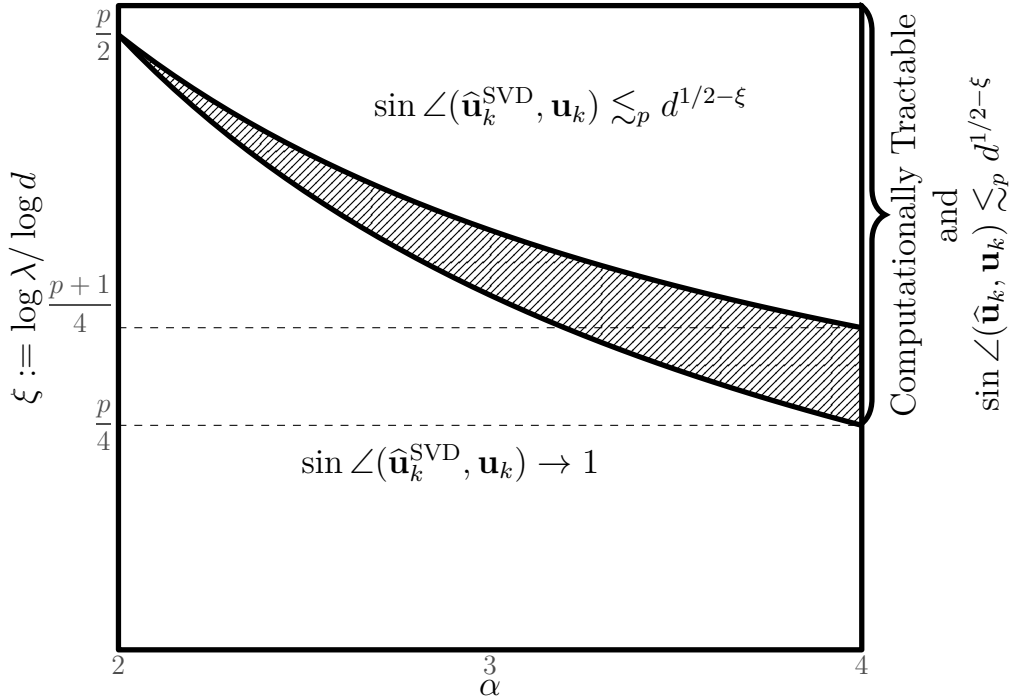
\begin{figure}[htbp]
\centering
\begin{tikzpicture}[x=.7pt,y=1pt,scale=0.75]
\definecolor{fillColor}{RGB}{255,255,255}
\begin{scope}
\definecolor{drawColor}{RGB}{255,255,255}
\definecolor{fillColor}{RGB}{255,255,255}

\path[draw=drawColor,line width= 0.6pt,line join=round,line cap=round,fill=fillColor] (  0.00,  0.00) rectangle (628.75,361.35);
\end{scope}
\begin{scope}
\definecolor{drawColor}{RGB}{0,0,0}

\path[draw=drawColor,line width= 2pt,line join=round] ( 60.93,341.07) --
	( 74.32,331.46) --
	( 87.71,322.30) --
	(101.10,313.57) --
	(114.49,305.24) --
	(127.88,297.28) --
	(141.26,289.66) --
	(154.65,282.37) --
	(168.04,275.38) --
	(181.43,268.68) --
	(194.82,262.24) --
	(208.21,256.06) --
	(221.59,250.11) --
	(234.98,244.39) --
	(248.37,238.89) --
	(261.76,233.58) --
	(275.15,228.46) --
	(288.54,223.52) --
	(301.93,218.75) --
	(315.31,214.14) --
	(328.70,209.69) --
	(342.09,205.38) --
	(355.48,201.21) --
	(368.87,197.18) --
	(382.26,193.27) --
	(395.64,189.48) --
	(409.03,185.80) --
	(422.42,182.24) --
	(435.81,178.78) --
	(449.20,175.42) --
	(462.59,172.15) --
	(475.98,168.98) --
	(489.36,165.90) --
	(502.75,162.90) --
	(516.14,159.98) --
	(529.53,157.14) --
	(542.92,154.37) --
	(556.31,151.68) --
	(569.70,149.05) --
	(583.08,146.50) --
	(596.47,144.00);

\path[draw=drawColor,line width= 2pt,line join=round] ( 60.93,341.07) --
	( 74.32,333.86) --
	( 87.71,326.99) --
	(101.10,320.45) --
	(114.49,314.20) --
	(127.88,308.22) --
	(141.26,302.51) --
	(154.65,297.04) --
	(168.04,291.80) --
	(181.43,286.78) --
	(194.82,281.95) --
	(208.21,277.31) --
	(221.59,272.85) --
	(234.98,268.56) --
	(248.37,264.43) --
	(261.76,260.45) --
	(275.15,256.61) --
	(288.54,252.91) --
	(301.93,249.33) --
	(315.31,245.88) --
	(328.70,242.54) --
	(342.09,239.30) --
	(355.48,236.18) --
	(368.87,233.15) --
	(382.26,230.22) --
	(395.64,227.38) --
	(409.03,224.62) --
	(422.42,221.95) --
	(435.81,219.35) --
	(449.20,216.83) --
	(462.59,214.38) --
	(475.98,212.00) --
	(489.36,209.69) --
	(502.75,207.44) --
	(516.14,205.25) --
	(529.53,203.12) --
	(542.92,201.05) --
	(556.31,199.03) --
	(569.70,197.06) --
	(583.08,195.14) --
	(596.47,193.27);
\definecolor{fillColor}{RGB}{190,190,190}

\path[fill=fillColor,pattern=north east lines] ( 60.93,341.07) --
	( 74.32,331.46) --
	( 87.71,322.30) --
	(101.10,313.57) --
	(114.49,305.24) --
	(127.88,297.28) --
	(141.26,289.66) --
	(154.65,282.37) --
	(168.04,275.38) --
	(181.43,268.68) --
	(194.82,262.24) --
	(208.21,256.06) --
	(221.59,250.11) --
	(234.98,244.39) --
	(248.37,238.89) --
	(261.76,233.58) --
	(275.15,228.46) --
	(288.54,223.52) --
	(301.93,218.75) --
	(315.31,214.14) --
	(328.70,209.69) --
	(342.09,205.38) --
	(355.48,201.21) --
	(368.87,197.18) --
	(382.26,193.27) --
	(395.64,189.48) --
	(409.03,185.80) --
	(422.42,182.24) --
	(435.81,178.78) --
	(449.20,175.42) --
	(462.59,172.15) --
	(475.98,168.98) --
	(489.36,165.90) --
	(502.75,162.90) --
	(516.14,159.98) --
	(529.53,157.14) --
	(542.92,154.37) --
	(556.31,151.68) --
	(569.70,149.05) --
	(583.08,146.50) --
	(596.47,144.00) --
	(596.47,193.27) --
	(583.08,195.14) --
	(569.70,197.06) --
	(556.31,199.03) --
	(542.92,201.05) --
	(529.53,203.12) --
	(516.14,205.25) --
	(502.75,207.44) --
	(489.36,209.69) --
	(475.98,212.00) --
	(462.59,214.38) --
	(449.20,216.83) --
	(435.81,219.35) --
	(422.42,221.95) --
	(409.03,224.62) --
	(395.64,227.38) --
	(382.26,230.22) --
	(368.87,233.15) --
	(355.48,236.18) --
	(342.09,239.30) --
	(328.70,242.54) --
	(315.31,245.88) --
	(301.93,249.33) --
	(288.54,252.91) --
	(275.15,256.61) --
	(261.76,260.45) --
	(248.37,264.43) --
	(234.98,268.56) --
	(221.59,272.85) --
	(208.21,277.31) --
	(194.82,281.95) --
	(181.43,286.78) --
	(168.04,291.80) --
	(154.65,297.04) --
	(141.26,302.51) --
	(127.88,308.22) --
	(114.49,314.20) --
	(101.10,320.45) --
	( 87.71,326.99) --
	( 74.32,333.86) --
	( 60.93,341.07) --
	cycle;

\draw [dashed] (60.93,193.27) -- (596.47,193.27);
\draw [dashed] (60.93,144) -- (596.47,144);

\end{scope}
\begin{scope}
\definecolor{drawColor}{RGB}{0,0,0}

\path[draw=drawColor,line width= 2pt,line join=round] ( 60.93, 30.69) --
	( 60.93,355.85) -- (596.47,355.85) -- (596.47,30.69) -- cycle;
\end{scope}
\begin{scope}
\definecolor{drawColor}{gray}{0.30}


\node[text=drawColor,anchor=base east,inner sep=0pt, outer sep=0pt, scale=  0.88] at ( 57.21,140.97) {$\dfrac{p}{4}$};


\node[text=drawColor,anchor=base east,inner sep=0pt, outer sep=0pt, scale=  0.88] at ( 57.21,200.97) {$\dfrac{p+1}{4}$};


\node[text=drawColor,anchor=base east,inner sep=0pt, outer sep=0pt, scale=  0.88] at ( 57.21,338.04) {$\dfrac{p}{2}$};
\end{scope}
%
%
%
%
%
%
%
%
%
%
\begin{scope}
\definecolor{drawColor}{gray}{0.30}

\node[text=drawColor,anchor=base,inner sep=0pt, outer sep=0pt, scale=  0.88] at ( 60.93, 18.68) {2};


\node[text=drawColor,anchor=base,inner sep=0pt, outer sep=0pt, scale=  0.88] at (328.70, 18.68) {3};


\node[text=drawColor,anchor=base,inner sep=0pt, outer sep=0pt, scale=  0.88] at (596.47, 18.68) {4};
\end{scope}
\begin{scope}
\definecolor{drawColor}{RGB}{0,0,0}

\node[text=drawColor,anchor=base,inner sep=0pt, outer sep=0pt, scale=  1.10] at (328.70,  7.64) {$\alpha$};
\end{scope}
\begin{scope}
\definecolor{drawColor}{RGB}{0,0,0}

\node[text=drawColor,rotate= 90.00,anchor=base,inner sep=0pt, outer sep=0pt, scale=  1.10] at (1,213.27) {$\xi:=\log\lambda/\log d$};

\node[text=drawColor,rotate= 90.00,anchor=base,inner sep=0pt, outer sep=0pt, scale=  1.10] at (635,250) {Computationally Tractable};
\node[text=drawColor,rotate= 90.00,anchor=base,inner sep=0pt, outer sep=0pt, scale=  1.10] at (665,250) {and};
\node[text=drawColor,rotate= 90.00,anchor=base,inner sep=0pt, outer sep=0pt, scale=  1.10] at (695,235.97) {$\sin\angle(\hat{\bu}_k,\bu_k)\lesssim_p d^{1/2-\xi}$};

\node[text=drawColor, anchor=base,inner sep=0pt, outer sep=0pt, scale=  1.10] at (300, 120) {$\sin\angle(\hat{\bu}_k^{\rm SVD},\bu_k)\to 1$};

\node[text=drawColor, anchor=base,inner sep=0pt, outer sep=0pt, scale=  1.10] at (380, 300) {$\sin\angle(\hat{\bu}_k^{\rm SVD},\bu_k)\lesssim_p d^{1/2-\xi}$};

\draw [decorate,decoration={brace,amplitude=10pt},xshift=-4pt,yshift=0pt, line width= 2pt] (600,356.5)-- (600,144) node [black,midway,xshift=-0.6cm]{};
\end{scope}
\end{tikzpicture}
\caption{Tradeoff between statistical and computational efficiencies in estimating spiked rank-one tensors when the noise does not have fourth moments.}
\label{fig:a2}
\end{figure}

Due to the suboptimality of tensor SVD, it is doubtful if power iteration would work when $\alpha<4$. To this end, we consider a different estimating strategy. More specifically, our techniques are based on recent developments in the theory of robust estimation of the mean in the presence of heavy tailed errors. These works derive estimators with subgaussian concentration, inspired from the pioneering work of \cite{catoni2012challenging}. The key idea is to reduce the adverse effect of heavy tails through an influence function, and can be extended to matrix estimation. For covariance matrix estimation, \cite{catoni2016pac} and \cite{mendelson2020robust} were some of the first works in this area, although both these approaches involved optimizing over a $d$-dimensional $\eps$-net and thus having exponential time complexity. \cite{avella2018robust} have similar results  with polynomial time, but they too require an extensive search for tuning parameters. We will instead use results on spectrum truncated estimators applied to covariance estimation. \cite{giulini2015pac} described one such method for robust PCA through smooth truncation, based on which \cite{minsker2018sub} and \cite{ke2019user} provided more tractable procedures and general results.

Our results are obtained by exploiting close connections between estimating the rank-one spikes and the spectral norm of a random tensor of iid entries. We show that the order of the spectral norm of a random tensor can be precisely characterized by the moment of its entries, which might be of independent interest. In particular, our result indicates that, up to a logarithmic factor, the norm of the random tensor $\|\scrE\|$ is of the order $\sqrt{d}$ if and only if its entries have finite $4(p-1)$th moment. This can be viewed as a generalization of the classical results for random matrices \citep[see, e.g.,][]{bai1988note, silverstein1989weak}. In deriving these bounds, we used techniques developed for random matrices by \cite{latala2005some} and improved moment bounds of random tensors established earlier by \cite{nguyen2015tensor}.

The rest of the paper is organized as follows. We first develop probabilistic bounds for the spectral norm of a random tensor of iid entries and use these tools to study the performance of the tensor SVD in Section 2. Polynomial time computable estimation schemes are given in Sections 3 and 4 for $\alpha\ge 4$ and for $\alpha\ge 2$ respectively. To corroborate our theoretical development, Section 5 provides simulation studies to further demonstrate the practical merits of the proposed methods. We conclude with a few remarks on the implications and future directions in Section 6. All proofs are relegated to Section 7.

\section{Tensor SVD and Spectral Norm of Random Tensors}\label{sec:tensor_norm}
The most natural approach to estimating the singular vectors is via the tensor SVD. In particular, let
\begin{equation}
\label{eq:defsvd}
(\hat{\bu}_k^{\rm SVD}: 1\le k\le p)=\argmax_{\ba_k\in \SS^{d-1}}\langle \mathscr{X}, \ba_1\otimes\cdots\otimes\ba_p\rangle,
\end{equation}
Here $\SS^{d-1}$ is the unit sphere in $\RR^d$. It is well known that the tensor SVD can be equivalently characterized the best rank-one apprpoximation to $\scrX$ in that
$$
(\hat{\lambda}^{\rm SVD}, \hat{\bu}_k^{\rm SVD}: 1\le k\le p)=\argmax_{\gamma\in \RR,\ba_k\in \SS^{d-1}} \|\scrX-\gamma\ba_1\otimes\cdots\otimes\ba_p\|_{\rm HS},
$$
where $\|\cdot\|_{\rm HS}$ is the Hilbert-Schmidt or Frobenius norm. See, e.g., \cite{zhang2001rank}. The performance of these singular vector estimates is closely related to the spectral norm of the noise tensor:
$$
\|\mathscr{E}\|:=\max_{\ba_k\in \SS^{d-1}} \langle \mathscr{E}, \ba_1\otimes\cdots\otimes\ba_p\rangle.
$$
For example, it is known that
\begin{equation}
\label{eq:daviskahan}
\max_{1\le k\le p}\sin\angle(\hat{\bu}_k^{\rm SVD}, \bu_k)\lesssim {\|\mathscr{E}\|\over \lambda},
\end{equation}
so that $\hat{\bu}_k^{\rm SVD}$s are consistent whenever $\lambda\gg \|\mathscr{E}\|$. See, e.g., \cite{auddy2020perturbation}. To this end, we shall first study the spectral norm of a random tensor consisting of independent and identically distributed entries.

\subsection{Norm of Random Tensors}
The problem of bounding the spectral norm of a random tensor is well-studied in the matrix case, i.e., $p=2$. In particular, \cite{silverstein1989weak} showed that if $\mathscr{E}$ is an iid ensemble, then $\|\mathscr{E}\|$ is of the order $\sqrt{d}$ if and only if its entries have finite (weak) fourth moment. In other words, when $p=2$ and the entries of $\mathscr{E}$ have finite fourth moment, $\hat{\bu}_k^{\rm SVD}$s are consistent if and only if $\lambda\gg \sqrt{d}$. A couple of questions arise naturally. First, is there similar characterization of $\|\mathscr{E}\|$ for higher order tensors? And what happens if the entries of $\mathscr{E}$ have only $\alpha$th moment for $2\le\alpha<4$? The next result aims to settle the first question.

\begin{theorem}
\label{th:moment4}
Let $\mathscr{E}\in \RR^{d\times\cdots\times d}$ be a $p$th order random tensor whose entries are independent copies of a random variable $E$ with mean zero and variance $\sigma^2$. Then there exists a constant $C_p>0$ depending on $p$ only such that for any $\alpha\ge 4$, $\EE|E|^\alpha<\infty$ implies that, with probability at least $1-d^{-\alpha/4+1}$,
$$
\|\mathscr{E}\|\le C_p\sigma \left[\sqrt{d\log d}+d^{\tfrac{p-1}{\alpha}+\tfrac{1}{4}}(\log d)^{\tfrac{3}{2}}\right].
$$
Conversely, there exists another constant $C_p'>0$ depending on $p$ only such that $\EE|E|^\alpha=\infty$ implies that
$$
\|\mathscr{E}\|\ge C'_p\sigma d^{\max\left\{\tfrac{1}{2},\tfrac{p-1}{\alpha}+\tfrac{1}{4}\right\}},\qquad {\rm almost\ surely}.
$$
\end{theorem}

The lower and upper bounds of Theorem \ref{th:moment4} match up to the logarithmic factor. In particular, $\|\mathscr{E}\|$ is of the order $\sqrt{d}$, up to a logarithmic factor, if and only if its entries have finite $4(p-1)$th moment. This can be viewed as a generalization of the classical result for $p=2$ from \cite{silverstein1989weak}. For higher order tensors ($p>2$), the precise order of $\|\mathscr{E}\|$ depends on the value of $\alpha$ for $4\le\alpha<4(p-1)$. Consider, for example, $p=3$. Then $\|\mathscr{E}\|$ is of the same order as that of an iid Gaussian ensemble, up to at most a logarithmic factor, as soon as $E$ has finite eighth moment. Yet, if $E$ only has finite $\alpha$th moment for $4\le\alpha<8$, then $\|\mathscr{E}\|$ depends on the exact value of $\alpha$, and decreases as $\alpha$ increases.

The next result complements Theorem \ref{th:moment4} and deals with the case when $2\le\alpha<4$.
\begin{theorem}
\label{th:moment2} Let $\mathscr{E}\in \RR^{d\times\cdots\times d}$ be a $p$th order random tensor whose entries are independent copies of a random variable $E$ with mean zero and variance $\sigma^2$. There exist constants $C_p, C_p'>0$ depending on $p$ only such that for any $2\le\alpha<4$, $\EE|E|^\alpha<\infty$ implies that, with probability at least $1-d^{-\alpha/2+1}$,
$$
\|\mathscr{E}\|\le C_p\sigma d^{\tfrac{p-1}{\alpha}+\tfrac{1}{2}}(\log d)^{\tfrac{3}{2}}.
$$
Conversely, if $\EE|E|^\alpha=\infty$ then
$$
\|\mathscr{E}\|\ge C'_p\sigma d^{\tfrac{p}{\alpha}},\qquad {\rm almost\ surely}.
$$
\end{theorem}

Note that there is a gap between the upper bound and lower bound in Theorem \ref{th:moment2} beyond the logarithmic factor. While it is plausible that this is the result of our proof technique, it remains a possibility that this may point to something more fundamental. 

\subsection{Convergence Rates for Tensor SVD}
In light of \eqref{eq:daviskahan}, Theorems \ref{th:moment4} and \ref{th:moment2} immediately imply the consistency of $\hat{\bu}_k^{\rm SVD}$s when
\begin{equation}
\label{eq:defcrit}
\lambda\gg \lambda_{\rm crit}(d;\alpha):=\left\{\begin{array}{ll}d^{1/2}(\log d)^{1/2}& {\rm if\ }\alpha>4(p-1)\\ d^{{\tfrac{p-1}{\alpha}+\tfrac{1}{4}}}(\log d)^{3/2}& {\rm if\ }4\le \alpha\le 4(p-1)\\ d^{{\tfrac{p-1}{\alpha}+\tfrac{1}{2}}}(\log d)^{3/2}& {\rm if\ }2\le\alpha<4\end{array}\right..
\end{equation}
In fact, under this condition of the signal-to-noise ratio, much stronger statement can be made and in fact, $\hat{\bu}_k^{\rm SVD}$s can be shown to be rate optimal:
\begin{theorem}\label{pr:svd}
	Let $\mathscr{E}\in \RR^{d\times\cdots\times d}$ be a $p$th order random tensor whose entries are independent random variable with mean zero, variance one, and finite $\alpha$th moment, e.g., $\EE|E|^\alpha<\infty$ for some $\alpha\ge 2$. Then there exist a numerical constant $C>0$ and another constant $C_p$ depending on $p$ only such that if $\lambda \ge C\lambda_{\rm crit}(d;\alpha)$, then
	$$
	\max_{1\le k\le p}\sin\angle(\hat{\bu}_k^{\rm SVD}, \bu_k)\le C_p{\sqrt{d}\over \lambda},
	$$
	with probability tending to one as $d$ increases.
\end{theorem}

For comparison, under Gaussian noise, $\hat{\bu}_k^{\rm SVD}$ converges to $\bu_k$ at the optimal rate of $\sqrt{d}/\lambda$ as soon as $\lambda >C\sqrt{d}$ for some constant $C>0$. Theorem \ref{pr:svd} shows that the same is true, up to a logarithmic factor, when the entries of $\mathscr{E}$ has finite $4(p-1)$th moment. However, when $\alpha<4(p-1)$, the rate $\sqrt{d}/\lambda$ can only be achieved when $\lambda$ is much larger than that required with Gaussian errors. Nonetheless the following result shows that when $\alpha>4$ these requirements are indeed optimal, up to a logarithmic factor, and therefore highlight a fundamental difference in behavior of tensor SVD with heavy tailed and Gaussian noise. 

\begin{theorem}
\label{th:lowersvd}
Let $\mathscr{E}\in \RR^{d\times\cdots\times d}$ be a $p$th order random tensor whose entries are independent copies of a random variable $E$ such that $\EE|E|^\alpha=\infty$ for some $4<\alpha<4(p-1)$ yet $\EE|E|^\beta<\infty$ for some $\beta>\max\{(p-2)\alpha/(p-1),4\}$. If $\lambda < C \max\{d^{\tfrac{(p-1)}{\alpha}+\tfrac{1}{4}},\,\sqrt{d}\}$ for any constant $C>0$, then for any constant $0<C_0<1$,
$$
\min_{1\le k\le p}\sin\angle(\hat{\bu}_k^{\rm SVD}, \bu_k)\ge C_0,
$$
with probability tending to one, as $d\to\infty$. Similarly, suppose that $\EE|E|^\alpha=\infty$ for some $2<\alpha<4$ and $\EE|E|^\beta<\infty$ for some $\beta>\max\{(2p-4)\alpha/(2p-\alpha),2\}$. If $\lambda < C d^{{p}/{\alpha}}$ for any constant $C>0$, then for any constant $0<C_0<1,$
$$
\min_{1\le k\le p}\sin\angle(\hat{\bu}_k^{\rm SVD}, \bu_k)\ge C_0
$$
with probability tending to one, as $d\to\infty$.
\end{theorem}

For concreteness, consider a continuous distribution symmetric about 0 whose survival function is given by
$$
\bar{F}(x):=1-F(x)=x^{-\alpha} L(x),\qquad {x>0}
$$
where $L(x)$ is slowly varying function at $+\infty$ in that $L(x)>0$ and
$$
\lim_{x\to\infty} {L(tx)\over L(x)}=1,\qquad t>0.
$$
For such distributions, $\alpha$ is often referred to as their tail index. It is clear that for $E\sim F$, $\EE(|E|^q)=\infty$ if and only if $q\ge \alpha$ . In light of Theorem \ref{th:lowersvd}, when $\alpha>4$, $\hat{\bu}_k^{\rm SVD}$ is inconsistent if $\lambda\lesssim \max\{d^{\tfrac{(p-1)}{\alpha}+\tfrac{1}{4}},\,\sqrt{d}\}$; when $2<\alpha<4$, $\hat{\bu}_k^{\rm SVD}$ is inconsistent if $\lambda\lesssim d^{{p}/{\alpha}}$. Conversely as a result of Theorem \ref{pr:svd}, $\hat{\bu}_k^{\rm SVD}$ converges to $\bu_k$ at the optimal rate if $\lambda\gtrsim \lambda_{\rm crit}(d;\alpha-\epsilon)$ for any $\epsilon>0$.





Interestingly, perhaps also surprisingly at the first sight, the inferior signal strength requirement for estimating the singular vectors under heavy-tailed noise is only a limitation of the tensor SVD and not a fundamental barrier in general. We now show that it is possible to improve the tensor SVD via a different estimation strategy at least when the signal-to-noise ratio is sufficiently high.

\section{Power Iteration with Spectral Initiation}\label{sec:mat_svd}
One of the chief challenges with the tensor SVD is the computational cost. It is well known that computing the best rank-one approximation
\eqref{eq:defsvd} is NP hard \citep[e.g.,][]{hackbusch2012tensor,hillar2013} so that it is infeasible to compute $\hat{\bu}^{\rm SVD}_k$s for large $d$. A common strategy to overcome this difficulty is to apply power iteration with spectral initialization, which has been shown to yield an estimator that is both polynomial time computable and rate optimal in the presence of Gaussian error. See, e.g., \cite{richard2014statistical, liu2017characterizing}. We shall now show that this strategy continues to work whenever $\alpha> 4$.

Recall that the first order condition yields that $\hat{\bu}^{\rm SVD}_k$s satisfies
$$
\mathscr{X}\times_{j\neq k}\hat{\bu}^{\rm SVD}_j\propto \hat{\bu}^{\rm SVD}_k,\qquad 1\le k\le p.
$$
Motivated by this property, we shall consider estimating $\bu_k$ through power iteration:
\begin{equation}
\label{eq:iter}
\bx_k^{[t+1]}={\mathscr{X}\times_{j\neq k}\bx_j^{[t]}\over \left\|\mathscr{X}\times_{j\neq k}\bx_j^{[t]}\right\|},
\end{equation}
with initial estimates $\bx_j^{[0]}$s. For this to work, we first need to be able to find a ``reasonably good'' initial estimate $\bx_j^{[0]}$ that can be efficiently computed. This is usually done through HOSVD.

More specifically, denote by ${\sf Mat}_k: \RR^{d\times\cdots\times d}\to \RR^{d\times d^{p-1}}$ the operator that collapses all indices other than the $k$th one of a $p$th order tensor and therefore converts it into a $d\times d^{p-1}$ matrix. Write
$$
\scrT=\lambda \bu_1\otimes\cdots\otimes\bu_p.
$$
It is not hard to see that
$$
{\sf Mat}_k(\scrT)=\lambda\bu_k(\bu_1\odot\cdots\odot\bu_{k-1}\odot\bu_{k+1}\odot\cdots\odot\bu_p)^\top,
$$
where $\odot$ stands for the Kronecker product so that we can estimate $\bu_k$ by the leading left singular vector, denoted by $\hat{\bu}_k^{\rm Mat}$, of ${\sf Mat}_k(\mathscr{X})$. Observe that
$$
\EE\left[{\sf Mat}_k(\mathscr{X}){\sf Mat}_k(\mathscr{X})^\top\right]=\lambda^2 \bu_k\otimes\bu_k+d^{p-1}I,
$$
and $\hat{\bu}_k^{\rm Mat}$ is the leading eigenvectors of ${\sf Mat}_k(\mathscr{X}){\sf Mat}_k(\mathscr{X})^\top-d^{p-1}I$. By Davis-Kahan Theorem, we have
\begin{eqnarray*}
\sin\angle(\bu_k,\hat{\bu}_k^{\rm Mat})&\le& {2\|{\sf Mat}_k(\mathscr{E}){\sf Mat}_k(\mathscr{E})^\top+{\sf Mat}_k(\mathscr{T}){\sf Mat}_k(\mathscr{E})^\top+{\sf Mat}_k(\mathscr{E}){\sf Mat}_k(\mathscr{T})^\top-d^{p-1}I\|\over \lambda^2}\\
&\le&{2\|{\sf Mat}_k(\mathscr{E}){\sf Mat}_k(\mathscr{E})^\top-d^{p-1}I\|+4\|{\sf Mat}_k(\mathscr{T}){\sf Mat}_k(\mathscr{E})^\top\|\over \lambda^2}.
\end{eqnarray*}
Following Bai-Yin's law, we then have
\begin{proposition}\label{pr:matricize4th}
Let $\mathscr{E}\in \RR^{d\times\cdots\times d}$ be a $p$th order random tensor whose entries are independent copies of a random variable $E$ with mean zero, variance one and $\EE|E|^\alpha<\infty$ for some $\alpha\ge 4$. Then
$$\max_{1\le k\le p}\sin\angle (\hat{\bu}_k^{\rm Mat},\bu_k)=O_p\left(\dfrac{d^{p/2}+\lambda \sqrt{d}}{\lambda^2}\right),\qquad {\rm as\ }d\to\infty.$$
\end{proposition}

Proposition \ref{pr:matricize4th} indicates that $\hat{\bu}_k^{\rm Mat}$s are consistent as soon as $\lambda\gg d^{p/4}$. It is worth comparing this requirement with that of $\hat{\bu}_k^{\rm SVD}$s: $\lambda\gg d^{\max\{1/2,1/4+(p-1)/\alpha\}}$. See Theorem \ref{th:lowersvd}. The former is more restrictive since $\alpha\ge 4$. As in the Gaussian noise case, this gap is likely a display of the tradeoff between computational and statistical efficiencies: $\hat{\bu}_k^{\rm Mat}$ is computationally tractable yet $\hat{\bu}_k^{\rm SVD}$ in general is not. On the other hand, the convergence rate for $\hat{\bu}_k^{\rm Mat}$ is inferior to that of $\hat{\bu}_k^{\rm SVD}$. However, we can improve upon $\hat{\bu}_k^{\rm Mat}$s by using $\hat{\bu}_j^{\rm Mat}$s in place of ${\bx}^{[0]}_j$s in \eqref{eq:iter} to get an updated estimate.

To see how this works, write
$$\bx^{[t]}_j=\sqrt{1-\rho_j^2} \bu_j + \rho_j \bv_j$$ 
where $\bv_j$ is a unit length vector perpendicular to $\bu_j$. Then
\begin{eqnarray*}
\scrX\times_{j\neq k} \bx^{[t]}_j&=&\lambda \left(\prod_{j\neq k} \sqrt{1-\rho_j^2}\right) \bu_k\\
&&+\sum_{A\subset ([p]\setminus \{k\})} \left(\prod_{j\in A} \sqrt{1-\rho_j^2}\right)\left(\prod_{j\notin A\cup\{k\}} \rho_j\right)\mathscr{E}\times_{j\in A}\bu_j\times_{j\notin A\cup\{k\}} \bv_j.
\end{eqnarray*}
Note that the second term on the righthand side can be bounded by, up to a constant, $\|\scrE\|$. In light of Proposition \ref{pr:matricize4th}, this implies that, if $\rho_j$ are uniformly bounded away from 1, then
$$
\sin\angle\left(\bx_k^{[t+1]},\bu_k\right)=O_p\left({\|\scrE\|\over \lambda}\right).
$$
In particular, in the case of Gaussian errors, $\|\scrE\|=O_p(\sqrt{d})$ so that we can conclude that
$$
\sin\angle\left(\bx_k^{[1]},\bu_k\right)=O_p\left({\|\scrE\|\over \lambda}\right)
$$
suggesting that a single iteration with $\bx_k^{[0]}=\hat{\bu}_k^{\rm Mat}$ ($k=1,\ldots,p$) leads to rate optimal estimates of $\bu_k$. The same technique can be applied whenever $\alpha>4(p-1)$ thanks to Theorem \ref{th:moment4}. The argument, however, breaks down when $\alpha<4(p-1)$ and a single iteration no longer suffices. Nonetheless, a more careful analysis shows that the performance keeps improving with more iterations and $O(\log d)$ number of iterations can yield a rate optimal of $\bu_k$s.

\begin{proposition}
\label{pr:iter4th}
Let $\mathscr{E}\in \RR^{d\times\cdots\times d}$ be a $p$th order random tensor whose entries are independent copies of a random variable $E$ with mean zero, variance one and $\EE|E|^\alpha<\infty$ for some $\alpha> 4$. There exist constants $C_1,C_2>0$ such that if $\lambda >C_1d^{p/4}$ and  $\rho^{[t]}<1$, then
$$
\rho^{[t+1]}\le C_2(\rho^{[t]})^2{\|\scrE\|\over \lambda}+O_p\left({\sqrt{d}\over \lambda}\right),
$$
where
$$
\rho^{[t]}=\max_{1\le k\le p}\sin\angle(\bx^{[t]}_k,\bu_k).
$$
\end{proposition}

In light of Propositions \ref{pr:matricize4th} and \ref{pr:iter4th}, we can estimate $\bu_k$ by running power iterations \eqref{eq:iter} with initialization
$$
\bx_k^{[0]}=\hat{\bu}_k^{\rm Mat},\qquad k=1,\ldots,p.
$$
And
$$
\max_{1\le k\le p}\sin\angle (\bx_k^{[T]},\bu_k)=O_p\left(\dfrac{\sqrt{d}}{\lambda}\right),\qquad {\rm as\ }d\to\infty,
$$
for $T\gtrsim \log d$ provided that $\lambda\ge Cd^{p/4}$ for a sufficiently large constant $C>0$. This proves that 

\begin{theorem}
\label{th:main4}
Assume that the entries of $\scrE$ are independent and identically distributed with zero mean, unit variance and finite $\alpha$th moment for some $\alpha>4$. There exist constants $C_1,C_2>0$ such that if $\lambda> C_1d^{p/4}$, then there is a polynomial time computable estimator $\hat{\bu}_k$ ($k=1,\ldots,p$) obeying
$$
\max_{1\le k\le p}\sin\angle(\hat{\bu}_k,\bu_k)\le {C_2\sqrt{d}\over \lambda},
$$
with probability tending to one as $d\to\infty$.
\end{theorem}

For this strategy to work we need $\lambda> \|\scrE\|$. However, in light of Theorem \ref{th:moment2}, this would require a higher signal-to-noise ratio than $d^{p/4}$ when $\alpha<4$. It turns out that while the vanilla power iteration may not work for smaller $\alpha$s, it is possible to attain both statistical and computational efficiencies as long as $\lambda\gtrsim d^{p/4}$ for any $\alpha\ge 2$.

\section{Tractable Estimation for all $\alpha\ge 2$}\label{sec:rob_svd}
As indicated in Theorem \ref{th:main4}, HOSVD and power iteration yields a consistent estimator under the signal strength requirement $\lambda\gtrsim d^{p/4}$ only if the entries of $\scrE$ have finite fourth moment. This can no longer be successful when $\lambda\sim d^{p/4}$ and $\alpha<4$, even without computational considerations, as shown by Theorem \ref{th:lowersvd}. To resolve this issue, we need to modify both the initialization and the power iteration steps. We first describe a new way for initialization.

\subsection{Initialization by Robust HOSVD}
The rationale behind the spectral initialization presented in the previous section is that ${\sf Mat}_k(\scrX){\sf Mat}_k(\scrX)^\top$ is an unbiased estimate of $\lambda^2\bu_k\otimes\bu_k$. However, this incurs bounding $\|{\sf Mat}_k(\mathscr{E}){\sf Mat}_k(\mathscr{E})^\top-I\|$ which requires finite fourth moment of the entries of $\scrE$. To relax this condition, we shall now proceed to estimate $\lambda^2\bu_k\otimes\bu_k$ via a more robust approach that works as long as the entries of $\scrE$ have finite variance.

In particular, we shall adopt a method first developed by \cite{catoni2012challenging} for estimating univariate mean, and later extended by \cite{minsker2018sub} for estimating matrices. It is based on an M-estimation framework where we estimate the common mean $\bM$ from some independent, but not necessarily identically distributed, samples $\bS_i, i=1,\ldots,n$ by
$$\hat{\bS}=\argmin_\bM\left[\tr \sum_{j=1}^n\Psi(\theta(\bS_j-\bM))\right],$$
and $\theta$ is a tuning parameter to be specified later. Here, for a function $f:\RR\to\RR$ and symmetric matrix $\bM$ with spectral decomposition $\bM=\bU\mathbf{\Lambda}\bU^\top$,
$$f(\bM)=\bU\diag(f(\lambda_1),\dots,f(\lambda_d))\bU^\top.$$
In particular, we shall take a $\Psi$ so that its first derivative $\psi=\Psi'$ is operator Lipschitz and obeys
$$-\log(1-x+x^2/2)\le\psi(x)\le \log(1+x+x^2/2).$$
See \cite{catoni2012challenging} and \cite{minsker2018sub} for further discussions and various examples.

Recall that
$$
{\sf Mat}_k(\scrX){\sf Mat}_k(\scrX)^\top=\sum_{i_{-k}\in[d]^{p-1}} \bX_{i_{-k}}\bX_{i_{-k}}^\top
$$
where $i_{-k}=(i_1,\ldots,i_{k-1},i_{k+1},\ldots,i_p)$ and $\bX_{i_{-k}}$ is the $k$th mode fiber of $\scrX$ with all indices except for the $k$th one fixed. Note that
$$
\EE(\bX_{i_{-k}}\bX_{i_{-k}}^\top)=\lambda^2 w_{i_{-k}}^2\bu_k\otimes\bu_k+I,
$$
where
$$
w_{i_{-k}}=\prod_{l\neq k} u_{li_l}.
$$
It is tempting to apply the aforementioned strategy directly to $\{\bX_{i_{-k}}\bX_{i_{-k}}^\top: i_{-k}\in [d]^{p-1}\}$ to estimate $\lambda^2\bu_k\otimes\bu_k$. There are, however, a couple of obstacles in doing so. Firstly, bounding the variation of $\hat{\bS}$ incurs the second moment of $\bS_j$s which can be translated into a requirement on the fourth moment of $\scrX$. This is exactly what we try to avoid. To this end, we shall instead consider estimating
$$
\bV_k=\lambda^2 \left[\bu_k\otimes\bu_k-{\rm diag}(\bu_k\otimes\bu_k)\right].
$$
Note that
$$
\left\|\bV_k-\lambda^2 \bu_k\otimes\bu_k\right\|=\lambda^2\|\bu_k\|_{\ell_\infty}^2.
$$
By Davis-Kahan Theorem, we know that
$$
\sin\angle(\bv,\bu_k)\le 2\|\bu_k\|_{\ell_\infty}^2.
$$
Therefore, by assuming that $2\|\bu_k\|_{\ell_\infty}^2<\eta$, a ``good'' estimate of the leading eigenvector of $\bV$ may yield an initial value satisfying the requirement of Proposition \ref{pr:iter}.

Another difficulty is that 
$$
\bY_{i_{-k}}=\bX_{i_{-k}}\bX_{i_{-k}}^\top-{\rm diag}(\bX_{i_{-k}}\bX_{i_{-k}}^\top)
$$
have different means. To this end, we randomly partition $[d]^{p-1}$ into $n$ groups, denoted by $I_1,\ldots, I_n$. This sampling is done through $d^{p-1}$ samples of $\mathrm{Multinomial}\left(n;\tfrac{1}{n},\dots,\tfrac{1}{n}\right).$ Let
$$
\hat{\bV}_k={1\over n\theta}\sum_{j=1}^n \psi(\theta\bS_j),
$$
where
\begin{equation}\label{eq:defSj}
	\bS_j=\sum_{i_{-k}\in I_j}\bY_{i_{-k}}.
\end{equation}
$\hat{\bV}_k$ can be viewed as a one-step gradient descent for computing $\hat{\bS}$ with initial value $0$.

Denote by
$$
\mu_1=\max_{i_{-k}\in [d]^{p-1}} |w_{i_{-k}}|, \qquad {\rm and}\qquad \mu_2=\max_{1\le k\le p}\|\bu_k\|_{\ell_\infty}.
$$
And write $\hat{\bv}_k$ the leading eigenvector of $\hat{\bV}_k$. Then

\begin{theorem}\label{th:robust1}
Assume that $\lambda>Cd^{p/4}(\log d)^{1/4}$ and $\mu_1\le C^{-1}(\log d)^{-1}$ for a sufficiently large constant $C>0$. If 
$$\theta=\sqrt{\dfrac{8\log(d)}{\lambda^4/n+d^p}},$$
then
$$
\max_{1\le k\le p}
\sin\angle(\hat{\bv}_k,\bu_k)\le 2\mu_2^2+\sqrt{\dfrac{32\log d}{n}}+\dfrac{(\lambda\sqrt{d}+d^{p/2})\sqrt{8\log d}}{\lambda^2}
$$
with probability at least $1-Cd^{-1}\log d-n\exp(-1/Cn\mu_1^2).$
\end{theorem}

The algorithm above effectively does a truncation around $\mathbf{0}$. It is natural that this causes significant bias and leads to a larger deviation term. With more gradient iterations, $\hat{\bV}^{(t)}_k$ becomes an increasingly better approximation to $\bV_k$ and reduces the second term of the deviation exponentially fast. We omit details since we intend to use this only for initialization and the performance guarantee given by Theorem \ref{th:robust1} is sufficient for our purpose.

The theoretical choice of the truncation parameter $\theta$ as given above, requires some knowledge of $\lambda$. If  we instead have some preliminary bounds on $\lambda$, we define $\theta_j$ as follows by the so-called Lepski method. Let $\calL=\{l\in\NN: \lambda_{\min}\le \lambda_{l}=2^l\lambda_{\min}\le 2\lambda_{\max}\}.$ For each $\lambda_{l}$ the corresponding truncated estimators $\hat{\bV}_{(l)}$ are defined as above. Then,
$$
\theta=\sqrt{\dfrac{8\log d}{\lambda_l^4/n+d^p}}
$$
and
$$
l^*=\min\left\{l\in\calL:\forall k\in \calL, \,k>l,\,\|\hat{\bV}_{(l)}-\hat{\bV}_{(k)}\|\le \frac{\lambda_k^4/n+d^p}{12n}
\right\}.
$$
Using results from \cite{minsker2018sub}, it can be shown that this scheme provides estimates that differ from Theorem \ref{th:robust1} only by a constant. Notice that in our case we can get a crude upper bound for $\lambda$ using the Frobenius norm of one of the tensor pieces. Moreover, our simulation results show that a fixed upper bound for $\lambda$ often suffices and we do not need to estimate it.

\subsection{One Step Power Iteration with Sample Splitting}
In light of Theorem \ref{th:robust1}, if $\lambda\sim d^{p/4}(\log d)^{1/4}$, then we can ensure that $\sin\angle(\hat{\bv}_k,\bu_k)\le \eta$ for some constant $\eta<1$ by we take $n=C\log d$. We shall now consider using them in the power iteration. As suggested by Proposition \ref{pr:iter4th}, for the accuracy to improve from iteration to iteration, it is important that we have $\lambda\gtrsim \|\scrE\|$. In light of Theorem \ref{th:moment2}, the requirement that $\lambda\sim d^{p/4}$ cannot ensure that is the case when $\alpha<4$. It turns out that this requirement is a mere consequence of the complicated nonlinear relationship between the singular vectors and $\mathscr{E}$ induced by the iterations. If the initial values $\bx^{[0]}_k$s are independent of $\scrX$, then running the power iteration \eqref{eq:iter} once would result in a rate optimal estimate.

\begin{proposition}
\label{pr:iter}
Assume that $\bx_k^{[0]}$s are independent of $\mathscr{X}$ and satisfy
$$
\max_{1\le k\le p}\sin\angle(\bx_k^{[0]}, \bu_k)\le \eta
$$
for some constant $\eta<1$. Then for any $0<\delta<1$,
$$
\max_{1\le k\le p}\sin\angle(\bx_k^{[1]}, \bu_k)\le \max\left\{{C\sqrt{d}/\delta^{1/\alpha}\over\lambda(1-\eta^2)^{(p-1)/2}},1\right\}
$$
with probability at least $1-\delta$.
\end{proposition}

Proposition \ref{pr:iter} immediately suggests a simple strategy to estimate $\bu_k$s when we observe, in addition to $\mathscr{X}$, another independent copy of it, denoted by $\tilde{\mathscr{X}}$: first apply robust tensor SVD to $\tilde{\mathscr{X}}$, and then update the estimated singular vectors using \eqref{eq:iter}. As a direct consequence of Theorem \ref{th:robust1} and Proposition \ref{pr:iter}, the resulting estimate $\tilde{\bu}_k$s satisfy:
\begin{equation}
\label{eq:two}
\max_{1\le k\le p}\sin\angle(\tilde{\bu}_k, \bu_k)\lesssim_p {\sqrt{d}\over \lambda}.
\end{equation}
if $\lambda\ge Cd^{p/4}\log d$ for a sufficiently large constant $C>0$.

Of course, we do not have another copy of $\scrX$. To overcome this obstacle, we randomly partition the tensor into two halves along its $p$-th mode. Denote the two halves of indices by $J_1$ and $J_2$. We use the tensor $\scrX_1,$ with indices $[d]^{p-1}\times J_1$ for nontrivial initialization, and $\scrX_2$ with indices $[d]^{p-1}\times J_2$ for iteration. It can be derived from the scaled Chernoff bound that
$$
\PP\left(\sum_{i\in J_1}u_{pi}^2\ge 0.25 \right) \le \exp\left(-1/16\mu_2^2\right).
$$
See, e.g., Theorems 1, 2 and the subsequent remarks of \cite{raghavan1988probabilistic}. Note that we can write 
$$
\scrX_1=\lambda\norm*{\bu_{p,J_1}} \bu_1\otimes\dots\otimes\bu_{p-1}\otimes {\bu_{p,J_1}\over \norm*{\bu_{p,J_1}} }.
$$

The last two equations imply that $\scrX_1$ has a signal strength of at least
$$0.5\lambda> Cd^{p/4}(\log d)^{1/4}.$$
Thus we can use Theorem \ref{th:robust1}, assuming all the incoherence conditions are satisfied, to get estimates $\hat{\bv}_k$ such that
\begin{equation}\label{eq:init}
	\max_{1\le k\le p-1}\sin\angle\left(\hat{\bv}_k,\,\bu_k\right)\le \eta,
\end{equation}
for some constant $\eta<1$. Following \eqref{eq:iter}, we can use $\bx^{[0]}_t=\hat{\bv}_k$ with the second tensor $\scrX_2$ to yield an improved estimate  of $\bu_k$, denoted by $\hat{\bu}_k$. Notice that $\hat{\bv}_k$ are independent of $\scrX_2.$ In light of Proposition \ref{pr:iter},
we get

\begin{theorem}
\label{th:main2}
Assume that the entries of $\scrE$ are independent and identically distributed with zero mean, unit variance and $\EE|E|^{\alpha}<\infty$ for some $\alpha\ge 2$. There exist constants $C_1,C_2, C_3>0$ such that if $\lambda> C_1d^{p/4}(\log d)^{1/4}$ and $\max_{1\le k\le p}\|\bu_k\|_{\ell_\infty}\le C_2(\log d)^{-1}$, then there is a polynomial time computable estimate $\hat{\bu}_k$ ($k=1,\ldots,p$) obeying
$$
\PP\left\{\max_{1\le k\le p}\sin\angle(\hat{\bu}_k,\bu_k)\le {C_3\sqrt{d}\over \lambda t}\right\}\ge 1-t^\alpha
$$
for any $0<t<1$.
\end{theorem}

Note that the additional requirement of $\max_{1\le k\le p}\|\bu_k\|_{\ell_\infty}\le C_2(\log d)^{-1}$ ensures that the singular vectors are not too concentrated on a few coordinates and therefore allows us to capture the signal even after the sample splitting. In the event that this is not the case, our task can be effectively reduced to a problem of lower order. To see this, assume, without loss of generality, that $u_{11}=\|\bu_1\|_{\ell_\infty}\gtrsim (\log d)^{-1}$. Denote by $\scrX_j$ the $j$th slice of $\scrX$ along its first mode. It is clear that
$$
\scrX_1=\tilde{\lambda}\bu_2\otimes\cdots\otimes\bu_p+\scrE_1,
$$
where
$$
\tilde{\lambda}=\lambda u_{11}\gtrsim d^{p/4}{\rm polylog}(d),
$$
by assumption. Note that the signal strengths $\lambda$ and $\tilde{\lambda}$ are of the same order up to the logarithmic factor. However, $\scrX$ is a $p$th order tensor and $\scrX_1$ is of order $(p-1)$. It is therefore conceivable that estimating the singular vectors of $\scrX_1$ could be easier because of the relative higher signal-to-noise ratio.

Finally, notice that the robust estimation method of the present section does not depend on $\alpha$, provided $\alpha\ge 2$. This allows the user to apply this method without any prior knowledge about the error distribution. The numerical experiments of Section 5 also support this claim. When the signal strength condition is satisfied, the performance of the robust estimators does not depend on the number of moments of the errors.

\section{Numerical Experiments}

To complement the theoretical developments, we also conducted several sets of numerical experiments. In the first set of simulation
we set $d=400,$ $\scrT=\lambda \bu_1\otimes \bu_2\otimes \bu_3+\mathscr{E},$ where $\lambda= 3d^{3/4}$ and $\bu_1,\bu_2,\bu_3$ were sampled uniformly from the unit sphere.  The elements of $\mathscr{E}$ are independently simulated from symmetrized and appropriately scaled Pareto distributions. More specifically, we generated $E_{ijk}=P_{ijk}R_{ijk}/\sqrt{\nu/(\nu-2)},$ where $P_{ijk}\sim \mathrm{Pareto}(\nu)$ and $R_{ijk}$s are i.i.d. Rademacher random variables. The rescaling was done to ensure the errors have unit variance. Note that $E_{ijk}$ has finite $\alpha$th moment if and only if $\nu>\alpha$. We therefore varied $\nu$ to simulate noises satisfying different moment conditions. We ran the algorithm in Section \ref{sec:rob_svd} with an initial guess of $3000$ for $\lambda$. Even though this is a huge overestimate, it does not affect the final results. For comparison, we also computed the na\"ive estimate based on HOSVD. The results from 1000 simulation runs for each value of $\nu$ are summarized in Figure \ref{fig:sim-pareto}. It can be observed that the robust method provides an estimate that is strongly correlated with the true vector $\bu_1$, irrespective of $\nu$. On the other hand, the na\"ive estimate is almost orthogonal to the signal direction for smaller values of $\nu$, but its performance improves as $\nu$ approaches 4, as predicted by Proposition \ref{pr:matricize4th}.

\input{sim-pareto.tex}

We next provide a numerical experiment to corroborate the signal strength requirements for consistent estimation. The setup is similar to before and we fixed $\nu=2.1$ and varied $d$ from $250$ to $500$. We took $\lambda=3d^\xi$ for $\xi=0.6,0.75,\dots,1.2$ to correspond to different signal strength. The result, again summarized from 1000 simulation runs, is presented in Figure \ref{fig:pareto-xi-dim}. It indicates that $\xi=3/4$ is indeed  the correct computational threshold. When $\xi<0.75$, neither of the methods is successful. However, as soon as $\xi$ reaches $0.75$, the robust SVD method from Section \ref{sec:rob_svd} is able to provide nontrivial estimates. The accuracy improves as $\xi$ increases further. On the other hand, the na\"ive estimator performs poorly for $\xi$ as large as $1.05$, where it has a very large variance, before transitioning to a better estimate at $\xi=1.2$.

\input{pareto-xi-dim.tex}

To investigate the possible effect of different error distributions or lack thereof, we also considered a simulation setting similar to the one used by \cite{ding2020estimating}. We fixed $d=400$ and set $\scrT=\lambda \bu_1\otimes \bu_2\otimes \bu_3+\mathscr{E},$ where $\lambda= 1.5d^{3/4}$ and $\bu_1,\bu_2,\bu_3$ are sampled uniformly from the unit sphere. The errors are independently distributed as $R_{ijk}X_{ijk}$ where $R_{ijk}$s are Rademacher random variables while 
$X_{ijk}=-\sqrt{\dfrac{d-\nu}{\nu}}$ with probability $\dfrac{\nu}{d}$
and $X_{ijk}=\sqrt{\dfrac{\nu}{d-\nu}}$ with probability $1-\dfrac{\nu}{d}.$ The distribution becomes lighter tailed as $\nu$ increases. The robust method still has better performance than the na\"ive one, even for much lighter tailed errors. We arbitrarily fixed the truncation parameter $\theta=0.2$ and used a single robust iteration with no sample splitting. As shown by \cite{ding2020estimating}, this error distribution can worsen the performance of elementwise truncation, however our experiment results, summarized from 1000 simulations in Figure \ref{fig:sim-mix}, confirms that this has no effect on the spectrum truncated estimators that we proposed.

\input{sim-mix.tex}

We also examined the effect of signal strength for this noise distribution. We fixed the mixture parameter $\nu=0.1$ and vary the dimension $d$ from $200$ to $450$, while setting $\lambda=1.5 d^\xi$. The results summarized from 1000 simulations is given in Figure \ref{fig:mix-xi-dim}. The observation is similar to before: the robust SVD method is successful whenever $\xi\ge 0.75$. The na\"ive estimator is almost orthogonal to the signal till $\xi=0.8$, then goes through a high variance phase at $\xi=0.85$, finally providing a nontrivial estimate only when $\xi=0.9$.

\bigskip

\input{mix-xi-dim.tex}

\section{Concluding Remarks}

In this paper, we study the problem of estimating the rank-one spikes in the presence of heavy tailed noises. Our contributions are three-fold. First we investigate the performance of estimates from tensor SVD, perhaps the most natural approach especially if we neglect the computational cost. Our results identify the signal strength requirement for the tensor SVD to yield rate-optimal estimates. (Nearly) matching lower bounds are also given to show that these requirements are optimal in the sense that the tensor SVD is necessarily inconsistent if the signal strength is below these requirement.

Our analysis of the tensor SVD exploits its close connection with the spectral norm of random tensors, and our second contribution is to establish upper bounds and (nearly) matching lower bounds for a tensor consisting of independent mean zero random variables. Our bounds pinpoint the connection between spectral norm of a random tensor and the moment condition for its entries.

Finally, we develop procedures for estimating the singular vectors under heavy tailed noises that are tractable in that they are polynomial time computable, practical in that they are easy to implement, and yields estimates that converge to the true parameter at the optimal rate. In particular, we show that similar to the case with Gaussian noise, a single power iteration with spectral initialization suffices if the entries of the noise have finite $4(p-1)$th moment. If the entries have finite fourth moment but infinite $4(p-1)$th moment, then we need to do $O(\log d)$ number of power iterations. If the entries do not have finite fourth moment, we need a different strategy. This new procedure combines robust matrix estimation and sample splitting, and can be shown as both tractable and rate optimal.


\section{Proofs}
\subsection{Moment Bounds for Random Tensors}
The proof of Theorems \ref{th:moment4} and \ref{th:moment2} uses Talagrand's concentration inequality for convex Lipschitz functions combined with estimates of higher order moments via Khintchine and Rosenthal inequalities. In particular, it relies on the following moment bound for random tensors which may be of independent interest.

\begin{theorem}\label{th:norm_upper} 
Let $\mathscr{E}\in \RR^{d\times\cdots\times d}$ be a $p$th order random tensor whose entries are independent such that $\EE E_{i_1\dots i_p}=0$ and $\EE E^2_{i_1\dots i_p}=\sigma^2_{i_1\dots i_p}.$ Then for any $q\ge 1,$ there is a constant $C_{p}$ depending only on $p$ such that 
\begin{eqnarray*}
(\EE\norm*{\mathscr{E}}^q)^{\tfrac{1}{q}}\le &&C_{p}\sqrt{d\log d}\left(1+\max_{i_1,\ldots, i_p\in[d]} \sigma_{i_1\dots i_p} \right)\\
&&+C_{p}(\log d)^{\tfrac{3}{2}}\left(\sum_{k=1}^p\EE\left(\max_{i_l\in[d],l\neq k}\sum_{i_k=1}^d\left( E_{i_1\dots i_p}^2-\sigma^2_{i_1\dots i_p}\right)\right)^{\tfrac{q}{2}}\right)^{\tfrac{1}{q}}.
\end{eqnarray*}
\end{theorem}

Note that we do not assume that the entries of $\mathscr{E}$ are identically distributed in Theorem \ref{th:norm_upper}. In fact, it follows directly that the upper bounds in Theorems \ref{th:moment4} and \ref{th:moment2} continue to hold if we have independent, but not necessarily identically distributed errors, as long as the moment conditions are satisfied. We opt for the current version of Theorems \ref{th:moment4} and \ref{th:moment2} for ease of exposition.

It is not hard to see that
$$\norm*{\mathscr{E}}\ge \underset{k\in[p]}{\max}\underset{i_l\in[d],l\neq k}{\max}\left(\sum_{i_k=1}^d E_{i_1\dots i_p}^2\right)^{1/2}.$$ 
This immediately suggests that
$$
(\EE\norm*{\mathscr{E}}^q)^{\tfrac{1}{q}}\gtrsim\left(\sum_{k=1}^p\EE\left(\max_{i_l\in[d],l\neq k}\sum_{i_k=1}^d E_{i_1\dots i_p}^2\right)^{\tfrac{q}{2}}\right)^{\tfrac{1}{q}}.
$$
The lower bound above matches the upper bound in Theorem \ref{th:norm_upper} up to the $\log d$ terms for any fixed $p$. Indeed, a close inspection of the proof of of Theorem \ref{th:norm_upper} indicates that the $\log d$ terms in the upper bound may be removed altogether with some stronger moment assumptions. The proof of Theorem \ref{th:norm_upper} relies on a scheme developed earlier by \cite{latala1997estimation} and is similar in spirit to that from \cite{nguyen2015tensor}. 

\begin{proof}[Proof of Theorem \ref{th:norm_upper}]  By the standard symmetrization argument and conditioning \citep[see, e.g., Lemma 5 of ][]{nguyen2015tensor},
	\begin{equation}\label{eq:symm_cond}
	\left(\EE(\norm*{\mathscr{E}})^q\right)^{1/q}\le \sqrt{2\pi}\EE_{\mathscr{E}}\left(\EE\left(\norm{\mathscr{H}}^q\big|\mathscr{E}\right)\right)^{1/q},
	\end{equation}
	where $\mathscr{H}$ is a $d\times\dots \times d$ tensor with entries $ H_{i_1\dots i_p}= E_{i_1\dots i_p}Z_{i_1\dots i_p}$, $Z_{i_1\dots i_p}\stackrel{iid}{\sim}N(0,1)$. We will first show that for any fixed tensor $\mathscr{E},$ $\mathscr{H}$ defined above satisfies
	\begin{equation}\label{eq:norm_main}
		\begin{split}
			(\EE\norm*{\mathscr{H}}^q)^{1/q}\le &\,\, C_p\sqrt{d\log d}(1+\max \sigma^2_{i_1\dots i_p})^{1\over 2}
			\\&+C_{p}(\log d)^{\tfrac{3}{2}}\left(\sum_{k=1}^p\left(\max_{i_l\in[d],l\neq k}\sum_{i_k=1}^d E_{i_1\dots i_p}^2-\sigma^2_{i_1\dots i_p}\right)^{q\over 2}\right)^{1\over q}.
		\end{split}
	\end{equation}
To this end, we shall use an $\eps$-net argument.

For any integer $L,$ write $S_L=\{0,1,\dots,2^{-L}\}.$ It follows from Lemma 10 of \cite{nguyen2015tensor} that the set $N_L=\{\bx\in \RR^d:\norm{\bx}\le 1,x_i^2\in S_L\}$ forms a $(1/2)$-net for $\SS^{d-1}$ by taking $L=\log d+c_0$ for some constant $c_0$. Now define the projections
$$\Pi_{=l}(\bx)_i=x_i\mathbbm{1}(x_i^2=2^{-l})\qquad {\rm and}\qquad \Pi_{<l}(\bx)_i=x_i\mathbbm{1}(x_i^2\ge 2^{-l}).$$
Let $N_{=l}=\Pi_{=l}(N_L)$ and $N_{<l}=\Pi_{<l}(N_L)$ for $1\le l\le L$. Note that for any $\bx\in N_L,$
$$\bx=\sum_{l=1}^L\Pi_{=l}(\bx)\qquad {\rm and} \qquad \sum_{m<l}\Pi_{=m}(\bx)=\Pi_{<l}(\bx).$$ 
Expanding the sum for each vector $\bx_j,$  we get
	$$\begin{aligned}\mathscr{H}\times_2\bx_2\dots\times_p\bx_p&=\sum_{l_1=1}^L\dots\sum_{l_p=1}^L\mathscr{H}\times_2\Pi_{l_2}(\bx_2)\dots\times_p\Pi_{l_p}(\bx_p)
	\\&=\sum_{k=2}^p\underset{\mathrm{argmax} \,l_i=k}{\sum_{l_k=1}^L}\sum_{\underset{i\neq k}{l_i\le l_j}}\mathscr{H}\times_2\Pi_{l_2}(\bx_2)\dots\times_p\Pi_{l_p}(\bx_p)
	\\&=\sum_{k=2}^p\sum_{l_k=1}^L\mathscr{H}\times_2\left(\sum_{l_2\le l_k}\Pi_{l_2}(\bx_2)\right)\dots\times_k\left(\Pi_{l_k}(\bx_k)\right) \dots\times_p\left(\sum_{l_p\le l_k}\Pi_{l_p}(\bx_p)\right)
	\\&=\sum_{k=2}^p\sum_{l=1}^L\mathscr{H}\times_2\Pi_{<l}(\bx_2)\dots\times_{k-1}\Pi_{<l}(\bx_{k-1})\times_k\Pi_l(\bx_k)\times_{k+1}\dots\times_p\Pi_{<l}(\bx_p).
	\end{aligned}$$
By triangle inequality,
	\begin{equation}\label{eq:sumterms}\begin{split}&\norm*{\mathscr{H}}^2 =\sup_{\bx_2,\dots,\bx_p\in \calS^{d-1}}\norm*{ \mathscr{H}\times_2\bx_2\dots\times_p\bx_p}^2\\
	&\le 2^{2p-2} \max_{\bx_2,\dots,\bx_p\in N_L}\norm*{ \mathscr{H}\times_2\bx_2\dots\times_p\bx_p}^2
	\\&\le 2^{2p-2}\max_{\bx_2,\dots,\bx_p\in N_L}\left[\sum_{k=2}^p\norm*{\sum_{l=1}^L \mathscr{H}\times_2\Pi_{<l}(\bx_2)\dots\times_{k-1}\Pi_{<l}(\bx_{k-1})\times_k\Pi_l(\bx_k)\times_{k+1}\dots\times_p\Pi_{<l}(\bx_p)}\right]^2
	\\&\le 2^{2p-2}p\sum_{k=2}^p\max_{\bx_2,\dots,\bx_p\in N_L}\norm*{\sum_{l=1}^L \mathscr{H}\times_2\Pi_{<l}(\bx_2)\dots\times_{k-1}\Pi_{<l}(\bx_{k-1})\times_k\Pi_l(\bx_k)\times_{k+1}\dots\times_p\Pi_{<l}(\bx_p)}^2.
	\end{split}\end{equation}
Because of symmetry, we shall focus on $k=2$ without loss of generality. To simplify notation, let us denote
	$$\bT_l(\bx_1,\dots,\bx_p)=\mathscr{H}\times_2\Pi_{=l}(\bx_2)\times_3\Pi_{<l}(\bx_3)\dots\times_p\Pi_{<l}(\bx_p).$$
	For any fixed $\bx_2,\dots,\bx_p,$ we have 
	$$\norm*{\displaystyle\sum_{l=1}^L\bT_l(\bx_2,\dots,\bx_p)}\le \displaystyle\sum_{l=1}^L\norm*{\bT_l(\bx_2,\dots,\bx_p)}.$$
	
	\noindent Note that $\bT_l\sim N\left(\mathbf{0},\mathrm{diag}(\sigma_{1l}^2,\dots,\sigma_{dl})^2\right),$ where 
	$$\sigma_{i_1l}=\displaystyle\sum_{i_2=1}^d\dots \sum_{i_p=1}^d E_{i_1\dots i_p}^2\Pi_{l}(\bx_{2})_{i_2}^2\dots \Pi_{<l}(\bx_p)_{i_p}^2.$$ 
In light of Lemma 7 of \cite{nguyen2015tensor}, $\norm*{\bT_l}=f_l(\bZ)$ for a standard Gaussian vector $\bZ$ with
	$$\norm*{f_l}_{Lip}^2\le\max_{i_1}\sum_{i_2\dots i_p} E_{i_1\dots i_p}^2\Pi_{l}(\bx_{2})_{i_2}^2\dots \Pi_{<l}(\bx_p)_{i_p}^2.$$
	Note that $\norm*{f_l}_{Lip}$ depends on $\bx_j$s. Moreover
	$$\EE\norm*{\bT_l}^2=\dsum_{i_1\dots i_p} E_{i_1\dots i_p}^2\Pi_{l}(\bx_{2})_{i_2}^2\dots \Pi_{<l}(\bx_p)_{i_p}^2.$$ By Talagrand's concentration inequality for Lipschitz functions,
	$$\PP(\norm*{\bT_l}\ge \sqrt{\EE\norm*{\bT_l}^2}+t\norm*{f_l}_{Lip})\le \exp(-t^2/2).$$
	By Lemma 4 of \cite{latala2005some}, $$|N_{=l}|<|N_{<l}|<\exp(C2^L(1+L-l)).$$ An application of the union bound yields
	$$\PP\left(\underset{\bx_2,\dots,\bx_p\in N_L}{\cup}\norm*{\bT}_l\ge \sqrt{\EE\norm*{\bT}_l^2(\bx)}+C(\sqrt{(p-1)2^l(1+L-l)}+t)\norm*{f_l}_{Lip}(\bx)\right)\le \exp(-t^2/2).$$
	Summing over $l$ for each fixed $\bx_j,$ and by union bound over $l,$
	\begin{equation}\label{eq:norm_union_bd}
	\PP\left(\underset{\bx_2,\dots,\bx_p\in N_L}{\cup}\sum_{l=1}^L\norm*{\bT}_l\ge \sum_{l=1}^L\sqrt{\EE\norm*{\bT}_l^2(\bx)}+C\sum_{l=1}^L(\sqrt{(p-1)2^l(1+L-l)}+t)\norm*{f_l}_{Lip}(\bx)\right)\le Le^{-t^2/2}.
	\end{equation}
	We bound the ``sum of expectations'' term as
	$$\begin{aligned}\left(\sum_{l=1}^L\sqrt{\EE\norm*{\bT_l}^2(\bx)}\right)^2
	&\le L\sum_{l=1}^L\EE\norm*{\bT_l}^2\\
	&\le L\sum_{l=1}^L\sum_{i_2=1}^d\Pi_l(\bx_2)_{i_2}^2\sum_{i_3\dots i_p}\prod_{j=3}^p\Pi_{<l}(\bx_j)^2_{i_j}\left(\sum_{i_1=1}^d E_{i_1\dots i_p}^2\right)
	\\&\le L\sum_{l=1}^L\sum_{i_2=1}^d\Pi_l(\bx_2)_{i_2}^2\underset{i_3,\dots,i_p}{\max}\sum_{i_1=1}^d E^2_{i_1\dots i_p}
	\\&\le L\underset{i_2,\dots,i_p}{\max}\sum_{i_1=1}^d E^2_{i_1\dots i_p}\sum_{i_2=1}^d\sum_{l=1}^L\Pi_l(\bx_2)_{i_2}^2
	\\&\le L\underset{i_2,\dots,i_p}{\max}\sum_{i_1=1}^d E^2_{i_1\dots i_p}\sum_{i_2=1}^d\bx_{2i_2}^2\le L\underset{i_2,\dots,i_p}{\max}\sum_{i_1=1}^d E^2_{i_1\dots i_p}.
	\end{aligned}$$ 
	In the above we have used the facts that $\norm*{\Pi_{<l}(\bx_j)}\le 1$ and $\bx_2=\sum\Pi_{l}(\bx_2).$ For the other term, since $\norm*{\Pi_{<l}(\bx_j)}\le 1,$  we have
	$$\begin{aligned}
		2^l\norm*{f_l}^2_{Lip}(\bx)
		=&2^l\max_{i_2}\Pi_l(\bx_2)_{i_2}^2\max_{i_1,i_3,\dots,i_p}\sum_{i_2=1}^d E_{i_1\dots i_p}^2\mathbbm{1}(\Pi_l(\bx_2)_{i_2}\neq 0)
		\\=&2^l\times 2^{-l}\times \max_{i_1,i_3,\dots,i_p}\sum_{i_2=1}^d( E_{i_1\dots i_p}^2-\sigma^2_{i_1\dots i_p}+\sigma^2_{i_1\dots i_p})\mathbbm{1}(\Pi_l(\bx_2)_{i_2}\neq 0)
		\\\le & \max_{i_1,i_3,\dots,i_p}\sum_{i_2=1}^d( E_{i_1\dots i_p}^2-\sigma^2_{i_1\dots i_p})\mathbbm{1}(\Pi_l(\bx_2)_{i_2}\neq 0)+\max_{i_1,\dots,i_p} \sigma^2_{i_1\dots i_p} \sum_{i_2=1}^d\mathbbm{1}(\Pi_l(\bx_2)_{i_2}\neq 0)
		\\\le & \max_{i_1,i_3,\dots,i_p}\sum_{i_2=1}^d( E_{i_1\dots i_p}^2-\sigma^2_{i_1\dots i_p})+2^{l}\max_{i_1,\dots,i_p}\sigma^2_{i_1\dots i_p}
	\end{aligned}$$
	and thus 
	$$
	2^{l/2}\norm*{f_l}_{Lip}\le \underset{i_1,i_3,\dots,i_p}{\max}\sqrt{\abs*{\dsum_{i_2=1}^d( E_{i_1\dots i_p}^2-\sigma^2_{i_1\dots i_p})}}+2^{l/2}\underset{i_1,\dots,i_p}{\max}\sigma^2_{i_1\dots i_p}.
	$$ 
	Now the deviation term in \eqref{eq:norm_union_bd} can be bounded as
	\begin{eqnarray*}
	&&\sum_{l=1}^L\sqrt{2^l(1+L-l)}\norm*{f_l}_{Lip}(\bx)\\
         &\le& \left(L^{3/2}\max_{i_1,i_3,\dots,i_p}\sqrt{\abs*{\sum_{i_2=1}^d( E_{i_1\dots i_p}^2-\sigma^2_{i_1\dots i_p})}}+\max_{i_1,\dots,i_p}\sigma^2_{i_1\dots i_p}\max_l\sqrt{1+L-l}\sum_l2^{l/2}\right)\\
         &\le& \left(L^{3/2}\max_{i_1,i_3,\dots,i_p}\sqrt{\abs*{\sum_{i_2=1}^d( E_{i_1\dots i_p}^2-\sigma^2_{i_1\dots i_p})}}+\sqrt{Ld}\max_{i_1,\dots,i_p}\sigma^2_{i_1\dots i_p}\right).
	\end{eqnarray*}
	Similarly,
	$$\dsum_{l=1}^Lt\norm*{f}_{Lip}(\bx)\le t\sqrt{L}\underset{i_1,i_3,\dots,i_p}{\max}\sqrt{\dsum_{i_2=1}^d E_{i_1\dots i_p}^2}\sqrt{\sum_l2^{-l}}.$$
	Now taking supremum over all $\bx_2,\dots,\bx_p\in N_L$ in equation \eqref{eq:norm_union_bd},
	\begin{eqnarray*}\underset{\bx_2,\dots,\bx_p\in N_L}{\sup}\left(\sum_{l=1}^L\norm*{\bT_l}\right)^2\ge L\max_{i_2,\dots,i_p}\sum_{i_1=1}^d E_{i_1\dots i_p}^2+C_pLd(1+\max_{i_1,\dots,i_p}\sigma^2_{i_1\dots i_p})\\
	+C_p(L^3+t^2L)\max_{i_1,i_3,\dots,i_p}\sum_{i_2=1}^d( E_{i_1\dots i_p}^2-\sigma^2_{i_1\dots i_p})\end{eqnarray*}
	with probability at most $\exp(-t^2).$ Taking expectation over the Gaussians, one obtains
	\begin{eqnarray*}\EE\left(\underset{\bx_2,\dots,\bx_p\in N_L}{\sup}\left(\sum_{l=1}^L\norm*{\bT_l}\right)^q\right)\le p^{q-1} C_{p}(Ld)^{q\over 2}(1+\max_{i_1,\dots,i_p} \sigma^2_{i_1\dots i_p})^{q\over 2}\\+p^{q-1}C_{p}L^{q\over 2}\left(\max_{i_1,i_3,\dots,i_p}\sum_{i_2=1}^d E_{i_1\dots i_p}^2-\sigma^2_{i_1\dots i_p}\right)^{q\over 2}.\end{eqnarray*}
	Summing over all terms in \eqref{eq:sumterms}, we have 
	\begin{eqnarray*}
	(\EE\norm*{\mathscr{H}}^q)^{1/q}&\le&C_p\sqrt{Ld}(1+\max_{i_1,\dots,i_p} \sigma^2_{i_1\dots i_p})^{1\over 2}\\&&+p^{(q-1)/q}C_{p}L^{\tfrac{3}{2}}\left(\sum_{i_k=1}^p\left(\max_{i_l\in[d],l\neq k}\sum_{i_k=1}^d E_{i_1\dots i_p}^2-\sigma^2_{i_1\dots i_p}\right)^{q/2}\right)^{1/q}.
	\end{eqnarray*}
	Plugging in this bound in \eqref{eq:sumterms} proves \eqref{eq:norm_main} since $L=C\log d+c_0$. Now taking expectation in \eqref{eq:symm_cond} finishes the proof.
\end{proof}

\subsection{Norm of Random Tensors}
We are now in a position to prove Theorems \ref{th:moment4} and \ref{th:moment2}.  Without loss of generality, we take $\sigma=1$.
\begin{proof}[Proof of Theorem \ref{th:moment4}] We begin with the upper bound.

	\noindent\textbf{Upper bound.} For any $t>d,$ and $\EE\abs*{E}^\alpha=\kappa<\infty$,
	$$\PP\left(\underset{i_1,i_3,\dots,i_p}{\max}\sum_{i_2}\left(E^2_{i_1\dots i_p}-1\right)>t\right)\le d^{p-1}\cdot\dfrac{\EE\abs*{\sum_{i_2}(E^2_{i_1\dots i_p}-1)}^{\alpha/2}}{t^{\alpha/2}}\le\dfrac{C_pd^{p-1+\frac{\alpha}{4}}\kappa}{t^{\alpha/2}}$$
	by Khintchine and Rosenthal inequalities respectively. This means
	$$\EE\underset{i_1,i_3,\dots,i_p}{\max}\dsum_{i_2=1}^dE^2_{i_1\dots i_p}\le d+C_pd^{\tfrac{2(p-1)}{\alpha}+\tfrac{1}{2}}.$$
	By Theorem \ref{th:norm_upper} with $q=2,$ $$\EE\norm*{\mathscr{E}}\le C\sqrt{d\log d}+ C_pd^{\tfrac{p-1}{\alpha}+\tfrac{1}{4}}(\log d)^{3/2}.$$
	Notice that we can get a constant $C>0$ such that
	$$
	\PP\left(\max \abs*{E_{i_1\dots i_p}}>Cd^{\tfrac{p-1}{\alpha}+\tfrac{1}{4}}\right)\le d^p\EE\abs*{E}^\alpha /C^\alpha d^{p-1+\tfrac{\alpha }{4}}=d^{1-\tfrac{\alpha}{4}}.
	$$
	
	It is well known that the function
	$f:\RR^{d^p}\to\RR$ given by $f(\mathrm{vec}(\mathscr{E}))=\|\mathscr{E}\|$ is convex and $1$-Lipschitz. Now using Talagrand's concentration inequality for convex Lipschitz functions \citep[see, e.g., Equation 1.4 of][]{ledoux2013probability} we obtain
	$$\PP\left(\abs*{\|\mathscr{E}\|-\EE\|\mathscr{E}\|}>C_pd^{\tfrac{p-1}{\alpha}+\tfrac{1}{4}}(\log d)^{3/2} \right)\le d^{1-\tfrac{\alpha}{4}} 
	$$ and thus the upper bound now follows.

Now consider the lower bound.

\noindent\textbf{Lower bound.}  It is clear that
$$\norm*{\mathscr{E}}^2\ge \underset{i_2,\dots,i_p}{\max}\displaystyle\sum_{i_1=1}^dE_{i_1\dots i_p}^2.$$
Thus, for any constant $C>0$,
$$
\PP\left(\norm*{\mathscr{E}}^2>d+C^2d^{\tfrac{2(p-1)}{\alpha}+\tfrac{1}{2}}\,\,\mathrm{i.o.}\right)
\ge \PP\left(\underset{i_2,\dots,i_p}{\max}\displaystyle\sum_{i_1=1}^dE_{i_1\dots i_p}^2-d>C^4d^{\tfrac{2(p-1)}{\alpha}+\tfrac{1}{2}}\,\,\mathrm{i.o.}\right)
$$
Notice that $\dsum_{i_1=1}^d(E^2_{i_1\dots i_p}-1)$ is a sum of independent mean zero random variables. Since $\EE\abs*{E_{i_1\dots i_p}}^{\alpha/2}=\infty,$  Corollary 2 of \cite{latala1997estimation} along with Khintchine inequalities imply that for any finite $d$, the random variables $$X_{i_2\dots i_p}=\dsum_{i_1=1}^{d}(E_{i_1i_2\dots i_p}^2-1)/\sqrt{d}$$ satisfy
$$\EE |X_{i_2\dots i_p}|^{\alpha/2}\asymp \max\{1,\,d^{1-\alpha/4}(\EE|E_{i_1i_2\dots i_p}^2-1|)^{\alpha/2}\}=\infty.$$
For $k=1,2,\dots,$ let $$B_k=\{i_2,\dots,i_p:2^{k-1}< i_2\le 2^k, 1\le i_3,\dots,i_p\le 2^k\}.$$
By Borel Cantelli theorem, it is enough to show that
$$\sum_k \PP\left(\text{there exists }i_2,\dots,i_p\in B_k\text{ s.t. }\abs*{X_{i_2\dots i_p}}\ge C\cdot 2^{2k(p-1)/\alpha}\right)=\infty.$$
In other words, we need
$$\dsum_k\left[1-\PP\left(\abs*{X_{i_2\dots i_p}}< C\cdot 2^{2k(p-1)/\alpha}\right)^{2^{k(p-1)}/2}\right]=\infty.$$ 
Again since $\EE\abs*{X_{i_2\dots i_p}}^{\alpha/2}=\infty,$ we have $\PP(|X_{i_1\dots i_p}|>t)\gtrsim t^{-\alpha/2}$ for large enough $t$, and hence
$$\sum_{k=1}^\infty 2^{k(p-1)}\PP\left( \abs*{X_{i_2\dots i_p}}\ge C\cdot 2^{2k(p-1)/\alpha}\right)=\infty.$$ 
We use the well known implication
\begin{equation}\label{eq:bclemma_ineq}
	\sum_k\left[1-(1-a_k)^{b_k}\right]<\infty\implies \sum_k a_kb_k<\infty\quad\text{ for }a_k\in[0,1],\,\,b_k\ge 0.
\end{equation}
Notice that \eqref{eq:bclemma_ineq} implies, for $d=2^k$ and any constant $C>0,$ tensors $\mathscr{E}$ of dimension $d\times d\times\dots\times d$ satisfies
$$\PP\left((\norm*{\mathscr{E}}^2-d)/\sqrt{d}\ge C\cdot d^{2(p-1)/\alpha}\text{ for infinitely many }d\right)=1.$$
The proof is now completed.
\end{proof}

\vspace{20pt}

The proof of Theorem \ref{th:moment2} follows a similar strategy.

\begin{proof}[Proof of Theorem \ref{th:moment2}] 
	
	\noindent\textbf{Upper bound.} Recall that  $\EE\abs*{E}^{\alpha}=\kappa<\infty.$ 
	By Markov inequality, $$\PP\left(\underset{i_1,i_3,\dots,i_p}{\max}\dsum_{i_2=1}^dE^2_{i_1\dots i_p}>t\right)\le d^{p-1}\cdot \EE\abs*{\dsum_{i_2=1}^dE^2_{i_1\dots i_p}}^{\alpha/2}/t^{\alpha/2}.$$
	By Khintchine's inequality, for independent Rademacher random variables $R_{i_2},$
	\begin{eqnarray*}
	\EE\abs*{\sum_{i_2=1}^dE^2_{i_1\dots i_p}}^{\tfrac{\alpha}{2}}&\le& C\cdot \EE\abs*{\sum_{i_2=1}^dR_{i_2}E_{i_1\dots i_p}}^{\alpha}\\
	&\le& C\cdot\max\left\{\left(\EE\abs*{\sum_{i_2=1}^dR_{i_2}E_{i_1\dots i_p}}^{2}\right)^{\tfrac{\alpha}{2}},d\EE\abs*{E}^{\tfrac{\alpha}{2}}\right\}\\
	&=&Cd^{\tfrac{\alpha}{2}},
\end{eqnarray*}
	where the second step is by Rosenthal's inequality. Consequently
	$$\PP\left(\underset{i_1,i_3,\dots,i_p}{\max}\dsum_{i_2=1}^dE^2_{i_1\dots i_p}>t\right)\le C\cdot d^{p-1+\tfrac{\alpha}{2}}/t^{\tfrac{\alpha}{2}},$$
	and so
	$$\EE\underset{i_1,i_3,\dots,i_p}{\max}\dsum_{i_2=1}^dE^2_{i_1\dots i_p}\le C_pd^{\tfrac{2(p-1)}{\alpha}+1}.$$
	Similar to before, by Theorem \ref{th:norm_upper} with $q=2,$ $$\EE\norm*{\mathscr{E}}\le C\sqrt{d\log d}+ C_pd^{\tfrac{p-1}{\alpha}+\tfrac{1}{2}}(\log d)^{3/2}.$$ Moreover, we can get a constant $C>0$ such that
	$$
	\PP\left(\max \abs*{E_{i_1\dots i_p}}>Cd^{\tfrac{p-1}{\alpha}+\tfrac{1}{2}}\right)\le d^p\EE\abs*{E}^\alpha /C^{\alpha}d^{p-1+\tfrac{\alpha }{2}}=d^{1-\tfrac{\alpha}{2}}.
	$$
	Again, using Talagrand's concentration inequality for convex Lipschitz functions we obtain
	$$\PP\left(\abs*{\|\mathscr{E}\|-\EE\|\mathscr{E}\|}>C_pd^{\tfrac{p-1}{\alpha}+\tfrac{1}{2}}(\log d)^{3/2} \right)\le d^{1-\tfrac{\alpha}{2}},
	$$ and the upper bound now follows.
	
	\bigskip
		
	\noindent\textbf{Lower bound.} We will show that for any constant $C>0,$  $$\PP\left(\norm*{\mathscr{E}}^2>Cd^{2p/\alpha}\text{ i.o.}\right)=1.$$	Clearly,
	$$\norm*{\mathscr{E}}\ge \underset{i_1,\dots,i_p}{\max}\abs*{E_{i_1\dots i_p}}.$$ For $k=1,2,\dots,$ let $$A_k=\{i_1,\dots,i_p:2^{k-1}< i_1\le 2^k, 1\le i_2,\dots,i_p\le 2^k\}.$$ By Borel Cantelli theorem, it is enough to show that $$\displaystyle\sum_k\PP\left(\text{there exists }i_1,\dots,i_p\in A_k\text{ s.t. }\abs*{E_{i_1\dots i_p}}\ge C\cdot2^{kp/\alpha}\right)=\infty.$$
	As before, we need
	$$\displaystyle\sum_k \left[1-\PP\left(\abs*{E_{i_1\dots i_p}}< C\cdot 2^{kp/\alpha}\right)^{2^{kp}/2}\right]=\infty.$$ Notice now that 
	$$\EE\abs*{E_{i_1\dots i_p}}^{\alpha}\le C^{\alpha}+\displaystyle\sum_k\EE(\abs*{E_{i_1\dots i_p}}^{\alpha}\mathbbm{1}_{C\cdot 2^{kp/\alpha}\le \abs*{E_{i_1\dots i_p}}\le C\cdot 2^{(k+1)p/\alpha}}).$$ Thus,
	$$\sum_{k=1}^{\infty} 2^{kp}\PP(C\cdot 2^{kp/\alpha}\le \abs*{E_{i_1\dots i_p}}\le C\cdot 2^{(k+1)p/\alpha})=\infty$$
	since $\EE\abs*{E_{i_1\dots i_p}}^{\alpha}=\infty.$ 
	The conclusion now follows from equation \eqref{eq:bclemma_ineq}.
\end{proof}

\subsection{Bounds for Tensor SVD}
We now turn our attention to bounds for the tensor SVD and prove Theorems \ref{pr:svd} and \ref{th:lowersvd}. The cases when $\alpha>4$ and $2<\alpha<4$ can be treated in an identical fashion and we shall focus on the case when $\alpha>4$ for brevity.
 
\begin{proof}[Proof of Theorem \ref{pr:svd}] 

Note that
	\begin{eqnarray}
		\hat{\lambda}&:=&\scrX\times_1\hat{\bu}_1^{\rm SVD}\dots\times_p\hat{\bu}_p^{\rm SVD}\nonumber\\
		&\ge& \scrX\times_1\bu_1\dots\times_p\bu_p\nonumber\\
		&=&\lambda+\mathscr{E}\times_1\bu_1\dots\times\bu_p\nonumber\\
		&\ge& \lambda-\sqrt{d}  \label{eq:svlow}
	\end{eqnarray}
	with probability at least $1-d^{-1}$. 
	
	Write $$\hat{\bu}^{\mathrm{SVD}}_j=\sqrt{1-\rho_j^2}\bu_j+\rho_j\bv_j$$ where $\|\bv_j\|=1$ and $\bv_j\perp \bu_j$, for $1\le j\le p$. Let $\rho:=\max_j|\rho_j|$. Using the upper bounds from Theorem \ref{th:moment4} for $k\ge 2$, we can derive that
	\begin{align*}
		\hat{\lambda}
		=&\scrX\times_1\hat{\bu}^{\mathrm{SVD}}_1\dots\times_p\hat{\bu}^{\mathrm{SVD}}_p	\\
		=&\lambda\prod_{j=1}^{p}\sqrt{1-\rho_j^2}+\sum_{A\subset [p]} \left(\prod_{j\in A} \sqrt{1-\rho_j^2}\right)\left(\prod_{j\notin A} \rho_j\right)\mathscr{E}\times_{j\in A}\bu_j\times_{j\notin A} \bv_j \\
		\le& \lambda\left(\prod_{j=1}^{p}\sqrt{1-\rho_j^2}\right)^{1/p}+
		\sum_{A\subset [p]} \left(\prod_{j\in A} \sqrt{1-\rho_j^2}\right)\left(\prod_{j\notin A} \rho_j\right)\norm*{\mathscr{E}\times_{j\in A}\bu_j}\\
		\le& (\lambda+\sqrt{d})(1-\sum_j\rho_j^2/2)+C_p\rho\sqrt{d}+C_p\sum_{k=2}^{p}\rho^k\left(d^{\tfrac{k-1}{\alpha}+\tfrac{1}{4}}+\sqrt{d}\right)(\log d)^{3/2}\\
		\le& (\lambda+\sqrt{d})(1-\rho^2/2)+C_p\rho\sqrt{d}+C_p\rho^2\left(d^{\tfrac{p-1}{\alpha}+\tfrac{1}{4}}+\sqrt{d}\right)(\log d)^{3/2} \\
		\le& \lambda + C_p\sqrt{d}+\left[C_p\left(d^{\tfrac{p-1}{\alpha}+\tfrac{1}{4}}+\sqrt{d}\right)(\log d)^{3/2}-\lambda/2\right]\rho^2
		 \numberthis\label{eq:svup}
	\end{align*}
with probability at least $1-d^{-1}-d^{1-\alpha/4}$. Note that
$$|\mathscr{E}\times_1\bu_1\dots\times_p\bu_p|\le \sqrt{d}\qquad {\rm and} \qquad \|\mathscr{E}\times_{k\neq j}\bu_k\|\le C\sqrt{d}$$ with probability at least $1-d^{-1}$, using Chebychev inequalities. We also use AM-GM inequality for the first term on the fourth line.
	
	We can get a sufficiently large constant $C_p>0$, such that if  $$\lambda>C_p\left(d^{\tfrac{p-1}{\alpha}+\tfrac{1}{4}}+\sqrt{d}\right)(\log d)^{3/2},$$ and $\rho^2>  \sqrt{d}/\lambda$ the last line of \eqref{eq:svup} is at most $\lambda-2\sqrt{d}$, thus contradicting \eqref{eq:svlow}. We thus have
	\begin{equation}\label{eq:rho1}		
		\rho^2\le \sqrt{d}/\lambda.
	\end{equation}
It is also clear from \eqref{eq:svup} that $\hat{\lambda}\le \lambda+C_p\sqrt{d}$, which combined with \eqref{eq:svlow} yields
$$
|\hat{\lambda}-\lambda|\le C_p\sqrt{d}.
$$
We will derive an improved upper bound on $\rho$ by using the first order condition on $\hat{\bu}_j^{\rm SVD}$. In particular, $(\hat{\bu}^{\rm SVD}_1,\dots,\hat{\bu}^{\rm SVD}_p)$ is a local minimum of the function
	$$
	F(\gamma,\ba_1,\dots,\ba_p)=\|\scrX-\gamma\ba_1\otimes\dots\otimes\ba_p\|_{\rm HS}^2
	$$
	for $\gamma\in \RR,\,\ba_j\in\SS^{d-1}$. Setting the derivative of the Lagrangian to zero, we have
	$$
	\scrX\times_{k\neq j}\hat{\bu}^{\rm SVD}_k=\hat{\lambda}\hat{\bu}^{\rm SVD}_j\quad\text{for }1\le j\le p.
	$$
	For $j=1$,
	\begin{align*}
		\|\lambda(\hat{\bu}^{\rm SVD}_1-\bu_1)\|=&\|(\lambda-\hat{\lambda})\hat{\bu}^{\rm SVD}_1+(\hat{\lambda}\hat{\bu}^{\rm SVD}_1-\lambda \bu_1)\| \\
		\le & |\hat{\lambda}-\lambda|+\|(\scrT+\mathscr{E})\times_{k\neq 1}\hat{\bu}^{\rm SVD}_k-\lambda \bu_1\| \\
		\le & C_p\sqrt{d}+\norm*{\lambda\big(\prod_{k\neq 1}\sqrt{1-\rho_k^2}-1\big)\bu_1}+\|\mathscr{E}\times_{k\neq 1}\hat{\bu}^{\rm SVD}_k\|. \numberthis\label{eq:focnd}
	\end{align*}
	Since $$\rho^2=\max_j\rho_j^2\le \sqrt{d}/\lambda$$ by \eqref{eq:rho1}, it is not hard to see that 
	$$
	\abs*{\prod_{k\neq 1}\sqrt{1-\rho_k^2}-1}\le 1-(1-\rho^2)^{(p-1)/2}\le C_p\sqrt{d}/\lambda.
	$$
	On the other hand, following \eqref{eq:svup}, we have
	\begin{align*}
		\|\mathscr{E}\times_{k\neq 1}\hat{\bu}^{\rm SVD}_k\|=&\norm*{\sum_{A\subset ([p]\setminus \{1\})} \left(\prod_{k\in A} \sqrt{1-\rho_j^2}\right)\left(\prod_{k\notin A\cup\{1\}} \rho_j\right)\mathscr{E}\times_{j\in A}\bu_j\times_{k\notin A\cup\{1\}} \bv_j} \\
		\le& \sum_{A\subset ([p]\setminus \{1\})} \left(\prod_{k\in A} \sqrt{1-\rho_j^2}\right)\left(\prod_{k\notin A\cup\{1\}} \rho_j\right)\norm*{\mathscr{E}\times_{j\in A}\bu_j} \\
		\le& C_p\sqrt{d}+C_p\rho\sqrt{d}+ C_p\sum_{k=2}^{p-1}\rho^k\left(d^{\tfrac{k-1}{\alpha}+\tfrac{1}{4}}+\sqrt{d}\right)(\log d)^{3/2} \\
		\le &C_p\sqrt{d}+C_p\rho^2\left(d^{\tfrac{p-1}{\alpha}+\tfrac{1}{4}}+\sqrt{d}\right)(\log d)^{3/2} \\
		\le &C_p\sqrt{d}+C_p\cdot\frac{\sqrt{d}}{\lambda}\cdot \left(d^{\tfrac{p-1}{\alpha}+\tfrac{1}{4}}+\sqrt{d}\right)(\log d)^{3/2} \\
		\le &C_p\sqrt{d},
	\end{align*}
with probability at least $1-d^{1-\alpha/4}$, once again using the upper bounds from Theorem \ref{th:moment4}. The last line uses the facts  $\rho^2\le\sqrt{d}/\lambda$ and $$\lambda> C\left(d^{\tfrac{p-1}{\alpha}+\tfrac{1}{4}}+\sqrt{d}\right)(\log d)^{3/2}$$ for a sufficiently large constant $C>0$. 

Plugging the last two bounds into \eqref{eq:focnd} above implies
$$
\sin\angle(\hat{\bu}^{\rm SVD}_1,\,\bu_1)\le \sqrt{2}\|\hat{\bu}^{\rm SVD}_1-\bu_1\|\le \frac{C_p\sqrt{d}}{\lambda}.
$$
The bounds for $j=2,\dots,p$ follow by an analogous argument.	
\end{proof}

\vspace{20pt}

\begin{proof}[Proof of Theorem \ref{th:lowersvd}] Consider for some constant $0<C_0<1$, a set of vectors 
$$
		\calS(C_0)=\{\bx_1,\dots,\bx_p:\bx_j=\sqrt{1-\rho_j^2}\bu_j+\rho_j\bv_j,\,\|\bv_j\|=1,\bv_j\perp\bu_j,|\rho_j|<C_0\}.
		$$ 
		By assumption, there exists a $\beta$ such that $\EE| E_{i_1\dots i_p}|^{\beta}<\infty$ and 
		$$
		\beta> \max\{(p-2)\alpha/(p-1),4\} 
		$$ and all $i_1\dots i_p\in [d].$
		Following the steps of \eqref{eq:svup}, for any  $\bx_1,\dots,\bx_p\in \calS(C_0)$,
		\begin{align*} &\,\mathscr{E}\times_1\bx_1\dots\times_p\bx_p \\
			=&\sum_{A\subset [p]} \left(\prod_{j\in A} \sqrt{1-\rho_j^2}\right)\left(\prod_{j\notin A} \rho_j\right)\mathscr{E}\times_{j\in A}\bu_j\times_{j\notin A} \bv_j \\
			\le &\sum_{k=0}^p{{p}\choose{k}}\max_{|A|=p-k}\prod_{j\notin A}\rho_j\norm*{\mathscr{E}\times_{j\in A}\bu_j}\\			
			\le & \abs*{\mathscr{E}\times_1\bu_1\dots\bu_p}+C_p\sum_{k=1}^{p-1}{{p}\choose{k}}C_0^k(\sqrt{d}+d^{\tfrac{k-1}{\beta}+\tfrac{1}{4}})(\log d)^{3/2}+C_0^p\norm*{\mathscr{E}}\\
			\le & C_p\sqrt{d}+C_pd^{\tfrac{p-2}{\beta}+\tfrac{1}{4}}(\log d)^{3/2}+C_0^p\norm*{\mathscr{E}}\numberthis\label{eq:max_err}
		\end{align*}
		with probability at least $1-d^{1-\beta/4}$. Once again, in the third inequality above, we have used the upper bounds from Theorem \ref{th:moment4} for $k\ge 2$. Notice that 
		$$|\mathscr{E}\times_1\bu_1\dots\times_p\bu_p|\le \sqrt{d}\qquad {\rm and}\qquad \|\mathscr{E}\times_{k\neq j}\bu_k\|\le C\sqrt{d}$$ with probability at least $1-d^{-1}$, using Chebychev inequalities. On the other hand, $\EE\abs*{E}^\alpha=\infty$ for some $\alpha<4(p-1)$. Then  by the lower bounds in Theorems  \ref{th:moment4}, for any constant $C>0,$ $$\norm*{\mathscr{E}}>Cd^{\tfrac{p-1}{\alpha}+\tfrac{1}{4}}$$ almost surely. We then have
		\begin{align*}
			&\sup_{\bx_j\in \calS(C_0)}\scrX\times_1\bx_1\dots\times_p\bx_p \\
			\le & \sup_{\bx_j\in \calS(C_0)}\scrT\times_1\bx_1\dots\times_p\bx_p
			+\sup_{\bx_j\in \calS(C_0)}\scrE\times_1\bx_1\dots\times_p\bx_p \\
			\le &\lambda +C_p\sqrt{d}+C_pd^{\tfrac{p-2}{\beta}+\tfrac{1}{4}}(\log d)^{3/2}+C_0^p\norm*{\mathscr{E}}\\
			\le &\lambda+C_pd^{\tfrac{p-1}{\alpha}+\tfrac{1}{4}}+C_0^p\|\scrE\| \\
			\le &C_pd^{\tfrac{p-1}{\alpha}+\tfrac{1}{4}}+C_0^p\|\scrE\|-2\lambda \\
			\le &\|\scrE\|-2\lambda, \numberthis\label{eq:inn_prod_upper2}
		\end{align*}
		since $\beta>(p-2)\alpha/(p-1)$ and  $\lambda<Cd^{\tfrac{p-1}{\alpha}+\tfrac{1}{4}}$. Again,
		\begin{align*}
			&\scrX\times_1\hat{\bu}_1^{\rm SVD}\times_2\hat{\bu}_2^{\rm SVD}\times_3\dots\times_p\hat{\bu}_p^{\rm SVD} \\
			=& \sup_{\bx_1,\dots,\bx_p\in \SS^{d-1}}\scrX\times_1\bx_1\dots\times_p\bx_p\ge\norm*{\mathscr{E}}-\lambda,
		\end{align*}
		which when compared to \eqref{eq:inn_prod_upper2} shows that the global maximizer $(\hat{\bu}_1^{\rm SVD},\dots,\hat{\bu}_p^{\rm SVD})\notin \calS(C_0)$. In particular, $\|\hat{\bu}_j^{\rm SVD}-\bu_j\|>C_0$ with probability at least $1-d^{1-\beta/4}$, for any $C_0<1.$
		
		The same proof goes through for the case $\alpha>4(p-1)$ provided there is a small enough constant $C$ such that $\lambda< C\sqrt{d}$. Similarly, the case where $2<\alpha<4$	can be proved through the upper and lower bounds from Theorem \ref{th:moment2}.
	\end{proof}

\subsection{Bounds for Spectral Initialization and Power Iteration}
We now consider polynomial time computable estimates when $\alpha\ge 4$ by establishing bounds for spectral initialization and power iterations.

\begin{proof}[Proof of Proposition \ref{pr:matricize4th}.] 
	Notice that ${\sf Mat}_k(\scrE)$ is a $d\times d^{p-1}$ matrix of i.i.d.  random variables with mean $0$ and variance $1$. Also,
	$$
	{\sf Mat}_k(\scrT){\sf Mat}_k(\scrE)^\top
	=\lambda \bu_k(\bu_1\odot\dots\odot \bu_{k-1}\odot\bu_{k+1}\odot\dots\odot\bu_p)^\top {\sf Mat}_k(\scrE)^\top
	=\lambda \bu_k(\bE')^\top,
	$$
	where $\bE'$ is a $d$ length vector with independent random variables $\EE\bE'_i=0,$ $\EE(\bE'_i)^2=1$ and $\EE(\bE'_i)^4=\kappa<\infty.$ Then
	$$
	\PP\left(\norm*{\bu_k(\bE')^\top}>2C\sqrt{d}\right)= \PP\left((\sum E_i^{\prime 2}-1)^2>4Cd^2\right)\le Cd\EE((E'_i)^4)/d^2\le d^{-1}.
	$$
	 
	By Bai-Yin's law, $\lambda_{\max}({\sf Mat}_k(\mathscr{E}))=d^{(p-1)/2}+\sqrt{d}+o(\sqrt{d})$, meaning
	$$
	\norm*{{\sf Mat}_k(\scrE){\sf Mat}_k(\scrE)^\top-d^{p-1}I_{d}}=\lambda_{\max}({\sf Mat}_k(\mathscr{E}))^2-d^{p-1}\le Cd^{p/2}
	$$
	almost surely. See, e.g., Theorem 2 of \cite{bai2008limit} and Theorem 5.31 of \cite{vershynin2010introduction}.
	 
	 Now using Davis-Kahan theorem,
	$$
	\begin{aligned}
	\sin\angle\left(\hat{\bu}_k^{\sf Mat},\,\bu_k\right)
	\le & \dfrac{2\norm*{{\sf Mat}_k(\scrE){\sf Mat}_k(\scrE)^\top-d^{p-1}I_{d}}+4\lambda\norm*{{\sf Mat}_k(\scrT){\sf Mat}_k(\scrE)^\top}}{\lambda^2}
	\\\le & C\cdot\dfrac{d^{p/2}+\lambda\sqrt{d}}{\lambda^2}.
	\end{aligned}
	$$
	with  probability at least $1-d^{-1}$. The proof for other modes follows similarly.
\end{proof}

\vspace{20pt}

\begin{proof}[Proof of Proposition \ref{pr:iter4th}] The proof is by induction on $t$. The basis step holds by some nontrivial initialization, for example through the matricization estimator of Proposition \ref{pr:matricize4th}. We now assume that the induction hypothesis holds for some $t> 0$ and prove the induction step for $t+1$.
	
	As before, we write $$\bx^{[t]}_j=\sqrt{1-\rho_j^2} \bu_j + \rho_j \bv_j$$ where $\bv_j$ is a unit length vector perpendicular to $\bu_j$. Then
	\begin{eqnarray}
		\scrX\times_{j\neq k} \bx^{[t]}_j&=&\lambda \left(\prod_{j\neq k} \sqrt{1-\rho_j^2}\right) \bu_k \nonumber\\
		&&+\sum_{A\subset ([p]\setminus \{k\})} \left(\prod_{j\in A} \sqrt{1-\rho_j^2}\right)\left(\prod_{j\notin A\cup\{k\}} \rho_j\right)\mathscr{E}\times_{j\in A}\bu_j\times_{j\notin A\cup\{k\}} \bv_j.\label{eq:rhoexp}
	\end{eqnarray}
	Notice that the entries of $\bE'=\mathscr{E}\times_{j\neq k}\bu_j$ are i.i.d. copies of a random variable $E''$ with $\EE(E'')=0$, $\Var(E'')=1$ and $\EE\abs*{E''}^{4}=\kappa<\infty$. By Chebyshev's inequality, for any $1\le k\le p$,
	$$
	\PP\left(\|\mathscr{E}\times_{j\neq k}\bu_j\|>C\sqrt{d}\right)= \PP\left((\sum E_i^{\prime 2}-1)^2>4Cd^2\right)\le d\,\EE((E'_i)^4)/Cd^2\le d^{-1}.
	$$
	Notice also that 
	\begin{align*}
	&\sum_{A\subset ([p]\setminus \{k\}),\,|A|\le p-2} \left(\prod_{j\in A} \sqrt{1-\rho_j^2}\right)\left(\prod_{j\notin A\cup\{k\}} \rho_j\right)\mathscr{E}\times_{j\in A}\bu_j\times_{j\notin A\cup\{k\}} \bv_j\\
	&\le \sum_{l=2}^{p-1}{{p}\choose{l}}(\rho^{[t]})^l\|\mathscr{E}\|\le C_p(\rho^{[t]})^2\|\mathscr{E}\|.
	\end{align*}
The last two inequalities together imply that
\begin{align*}
	&\norm*{\mathscr{E}\times_{j\neq k}\bx_j^{[t]}}\\
	=&\norm*{\sum_{A\subset ([p]\setminus \{k\})} \left(\prod_{j\in A} \sqrt{1-\rho_j^2}\right)\left(\prod_{j\notin A\cup\{k\}} \rho_j\right)\mathscr{E}\times_{j\in A}\bu_j\times_{j\notin A\cup\{k\}} \bv_j} \\
	\le &C\sqrt{d}+C_p(\rho^{[t]})^2\|\mathscr{E}\|. \numberthis\label{eq:enorm}
\end{align*}
By the nontrivial initialization and the induction hypothesis, we have a constant $\rho^*<1$ such that $\rho^{[t]}\le \rho_*<1$. We then have
\begin{align*}
	\sin\angle(\bx_k^{[t]},\,\bu_k)=&\underset{\|\bw\|=1,\bw\perp \bu_k}{\sup}{{\scrX\times_{j\neq k}\bx_j^{[t]}\times_k\bw}\over{\|\scrX\times_{j\neq k}\bx_j^{[t]}\|}} \\
	=&\underset{\|\bw\|=1,\bw\perp \bu_k}{\sup}{{\mathscr{E}\times_{j\neq k}\bx_j^{[t]}\times_k\bw}\over{\|\scrX\times_{j\neq k}\bx_j^{[t]}\|}} \\
	\le & {\norm*{\mathscr{E}\times_{j\neq k}\bx_j^{[t]}}\over{\lambda \left(\prod_{j\neq k} \sqrt{1-\rho_j^2}\right) -\norm*{\mathscr{E}\times_{j\neq k}\bx_j^{[t]}}}}	\\
	\le & \dfrac{C\sqrt{d}+C_p(\rho^{[t]})^2\|\mathscr{E}\|}{\lambda(1-\rho_*^2))^{(p-1)/2}-C\sqrt{d}-C_p(\rho^{[t]})^2\|\mathscr{E}\|} \\
	\le & C\dfrac{\sqrt{d}}{\lambda}+C_p(\rho^{[t]})^2\dfrac{\|\mathscr{E}\|}{\lambda}
\end{align*}
with probability at least $1-d^{1-\alpha/4}$.

We use \eqref{eq:rhoexp} and \eqref{eq:enorm} for the first and second inequalities respectively. The last line follows if $\lambda >C\|\mathscr{E}\|$ for a sufficiently large constant $C>0$. Since we have $\lambda>d^{p/4}$ and $\alpha>4$, this condition is satisfied with probability at least $1-d^{1-\alpha/4}$, using the upper bounds from Theorem \ref{th:moment4}.
\end{proof}

\subsection{Bounds for Robust Tensor SVD}

\begin{proof}[Proof of Theorem \ref{th:robust1}.] 
	Let us fix $k=1$ as the other modes follow by symmetry.	
%

We will denote the partition of $[d]^{p-1}$ into $n$ groups as $I:=(I_1,\dots,I_{n}).$ Let us also define 
$$
\sigma^2=\norm*{\dfrac{1}{n}\sum_{j=1}^n\EE(\bS_j^2|I)}.
$$
Then conditional on $I$, we apply Theorem 3.2 of \cite{minsker2018sub} with $$\theta=\dfrac{\sqrt{2\log(2d/\delta)/n}}{\sigma}$$ to obtain
\begin{equation}\label{eq:one_step_bound}
\PP\left(\norm*{\hat{\bV}_1-\dfrac{1}{n}\sum_{j=1}^n\EE\left(\bS_j|I\right)}> \sigma \sqrt{\dfrac{2\log(2d/\delta)}{n}}\bigg| I\right)\le \delta.
\end{equation}

\bigskip

\noindent We need to calculate $\EE\left(\bS_j\bigg|I\right)$ and  $\sigma^2.$ We have the following lemma, the proof of which can be found in section \ref{subsec:lems}. \begin{lemma}\label{lem:mat_means} 
	$\EE(\bS_j|I)=\left(\dsum_{i_{-1}\in I_j}\bw_{i_{-1}}^2\right)\bV_1$ and
	\begin{equation}\label{eq:sgbd1}
		\sigma^2=\norm*{\frac{1}{n}\sum_{j=1}^n\EE(\bS_j^2|I)}\le \lambda^4\left(\dfrac{1}{n}\sum_{j=1}^n\left(\sum_{i_{-1}\in I_j}\bw^2_{i_{-1}}\right)^2\right)+\frac{\lambda^2d}{n}\sum_{j=1}^n\sum_{i_{-1}\in I_j}\bw_{i_{-1}}^2+\frac{d}{n}\sum_{j=1}^n|I_j|.
	\end{equation}
\end{lemma}

To complete the proof, we now use the multinomial sample splitting scheme   to get high probability bounds on the above quantities. We write $G_j=\displaystyle\sum_{i_{-1}\in I_j}\bw^2_{i_{-1}}.$ By our sampling scheme, 
$\mathbbm{1}(i_{-1}\in I_j)\sim\mathrm{Bernoulli}({1\over{n}})$ so that $\EE G_j=\tfrac{1}{n}.$ Then, the scaled Chernoff bound \citep[see, e.g., Theorems 1 and 2 and the subsequent remark of][]{raghavan1988probabilistic} along with the definition of $\mu_1$ yields,
\begin{equation}\label{eq:w4max1}
	\PP\left(\max_{1\le j \le n}G_j>\frac{2}{n}\right)
	\le n\PP\left(\dfrac{|G_1-\EE G_1|}{\underset{i_{-1}\in I_1}{\max} \|\bw_{i_{-1}}\|^2}
	>\dfrac{\EE G_1}{\underset{i_{-1}\in I_1}{\max} \|\bw_{i_{-1}}\|^2}\right)
	\le n\exp\left(-\dfrac{1/n}{\mu_1^2}\right).
\end{equation}
By the sample partition scheme $$\dsum_{j=1}^n|I_j|=d^{p-1}\qquad {\rm and} \qquad \dsum_{j=1}^n\sum_{i_{-1}\in I_j}\bw^2_{i_{-1}}=\norm*{\bw}^2=1.$$ Using \eqref{eq:sgbd1} and \eqref{eq:w4max1} we then have 
\begin{equation}\label{eq:sgbd2}
	\sigma^2\le \dfrac{2\lambda^4}{n^2}+\dfrac{\lambda^2d}{n}+\dfrac{d^{p}}{n}
\end{equation}
with probability at least $1-2n\exp\left(-\dfrac{1}{2n\mu_1^2}\right).$ Notice also that the signal matrix is 
$$
\frac{1}{n}\sum_{j=1}^n\EE(\bS_j|I)
={1\over n}\left(\sum_j\sum_{i_{-1}\in I_j}\bw^2_{i_{-1}}\right)\bV_1
={\lambda^2\over n}(\bu_k\bu_k^\top-\mathrm{diag}(\bu_k\bu_k^\top)).
$$
Now applying \eqref{eq:one_step_bound} with $\delta=1/d$, together with the noise bound from \eqref{eq:sgbd2} and using Davis-Kahan theorem, we have
\begin{align*}
	\sin\angle(\hat{\bv}_k,\,\bu_k)
	\le & 2\|\bu_k\|_{\ell_\infty}^2+\dfrac{2\sigma}{\lambda^2/n}\sqrt{\dfrac{2\log d}{n}}
	\\\le & 2\mu_2^2+2\left(2+\dfrac{(\lambda\sqrt{d}+d^{p/2})\sqrt{n}}{\lambda^2}\right)\sqrt{\dfrac{2\log d}{n}}
	\\\le& 2\mu_2^2+4\sqrt{\dfrac{2\log d}{n}}+\dfrac{(\lambda\sqrt{d}+d^{p/2})\sqrt{8\log d}}{\lambda^2}
\end{align*}
with probability at least $1-2n\exp\left(-\dfrac{1}{2n\mu_1 ^2}\right)-\dfrac{1}{d}.$
\end{proof}

\vspace{20pt}

%
\begin{proof}[Proof of Proposition \ref{pr:iter}] 
	Since $\bx_k^{[0]}$ are unit vectors that are independent of $\mathscr{E},$ $\bE=\mathscr{E}\times_{j\neq k}\bx_j^{[0]}$ is a $d\times 1$ vector whose entries are independent random variables with $\EE E_{i}=0$ and $\EE E_{i}^2=1$. Moreover, $\EE |E_i|^\alpha <\infty$ by Rosenthal inequality. Thus
	$$
	\PP\left(\sup_{\bv:\norm*{\bv=1}}\mathscr{E}\times_{j\neq k}\bx_j^{[1]}\times_k\bv>C\sqrt{d}/\delta^{1/\alpha}\right)
	\le \PP\left(\norm*{\bE}>C\sqrt{d}/\delta^{1/\alpha}\right)
	\le \dfrac{(\sum_i\EE E_i^2)^{\alpha/2}}{Cd^{\alpha/2}/\delta}=\delta.$$
	where we use Rosenthal inequalities in the last step.
	Therefore
	$$\begin{aligned}
		\sin\angle\left(\bx_k^{[1]},\bu_k\right)
		= &
		\underset{\bv:\|\bv\|=1,\,\bv\perp \bu_k}{\sup}|\langle\bx_k^{[1]},\bv\rangle |
		\le  
		\underset{\bv:\|\bv\|=1,\,\bv\perp \bu_k}{\sup} \dfrac{\mathscr{E}\times_{j\neq k}\bx_j^{[0]}\times_k\bv}
		{\|\mathscr{X}\times_{j\neq k}\bx_j^{[0]}\|}
		\\\le &
		\dfrac{\norm*{\mathscr{E}\times_{j\neq k}\bx_j^{[0]}}}{\lambda\displaystyle\prod_{j\neq k}\abs*{\langle \bx_j^{[0]},\bu_j\rangle }-\norm*{\mathscr{E}\times_{j\neq k}\bx_j^{[0]}\times_k\bu_k}}
		\\\le & \dfrac{C\sqrt{d}/\delta^{1/\alpha}}{\lambda(1-\eta^2)^{(p-1)/2}-\sqrt{d}/\delta^{1/\alpha}}
		\le\dfrac{C\sqrt{d}/\delta^{1/\alpha}}{\lambda(1-\eta^2)^{(p-1)/2}}.
	\end{aligned}$$
	with probability at least $1-\delta,$ provided  $\lambda(1-\eta^2)^{(p-1)/2}\ge 2\sqrt{d}/\delta^{1/\alpha}$.
\end{proof}

\vspace{20pt}
\begin{proof}[Proof of Theorem \ref{th:main2}] Note that the vector $\bu_{p,J_2}$ is the $p$-th mode of $\scrX_2$. By the scaled Chernoff bounds \citep[see, e.g., Theorems 1, 2 and subsequent remark of ][]{raghavan1988probabilistic},
$$
\PP\left(\|\bu_{p,J_2}\|<\dfrac{1}{2}\right)
=\PP\left(\abs*{\sum_{i=1}^d\left((\bu_p)_i\right)^2\mathbbm{1}(B_i=0)-\dfrac{1}{2}}>\dfrac{1}{4}\right)\le  \exp\left(-\dfrac{1/16}{\underset{i}{\max}\left((\bu_p)_i\right)^2}\right)\le Cd^{-1}
$$
so that $\|\bu_{p,J_2}\|\ge 0.5$ with high probability. By Theorem \ref{th:robust1} we have initializations $\bv_k$ independent of $\scrX_2$, satisfying \eqref{eq:init} for some constant $\eta<1$.

We immediately have from proposition \ref{pr:iter} that
$$
\PP\left(\sin\angle(\hat{\bu}_{p, J_2},\bu_{p,J_2}/\|\bu_{p,J_2}\|)
\le \dfrac{C\sqrt{d}}{\lambda t}\right)\ge 1-t^{\alpha}.
$$

\noindent We consider the first mode $\bu_1$ next. Since the unit initialization vectors $\hat{\bv}_k$ are independent of $\scrX_2$ for $k=2,\dots,p-1$ the matrix $\bE=\mathscr{E}_{2}\times_2\hat{\bv}_k\dots\times_{p-1}\hat{\bv}_{p-1}$ satisfies $\bE_{ij}$ are independent, $\EE\bE_{ij}=0,\,\EE\bE_{ij}^2=1.$ By Rosenthal inequalities, $\EE|E_{ij}|^\alpha<\infty$.

Moreover $\bw_2=\bE\bu_{p,J_2}/\|\bu_{p,J_2}\|$ again has independent entries with the same properties. By Theorem \ref{th:norm_upper} and Talagrand's concentration inequality, we have 
$$
\PP\left(\|\bE\|>Cd^{\tfrac{1}{2}+\tfrac{1}{\alpha\mathbbm{1}(\alpha <4)}}/t\right)\le t^{\alpha}.
$$
Under this event,
$$
\begin{aligned}
\underset{\bv:\|\bv\|=1,\,\bv\perp \bu^{(1)}}{\sup}& \mathscr{E}_{2}\times_1\bv\times_2\hat{\bv}_2\dots \times_{p-1}\hat{\bv}_{p-1}\times_p\hat{\bu}_{p, J_2}
\\&\le \|\bE(\bu_{p,J_2}/\|\bu_{p,J_2}\|)\|+\|\bE\|\|\hat{\bu}_{p, J_2}-(\bu_{p,J_2}/\|\bu_{p,J_2}\|)\|
\\&\le \|\bw_2\|+\|\bE\|\cdot\dfrac{C\sqrt{d}}{\lambda}
\\&\le C\sqrt{d/t^2}+C\cdot\dfrac{d^{\tfrac{1}{2}+\tfrac{1}{\alpha\mathbbm{1}(\alpha <4)}}}{t}\cdot\dfrac{\sqrt{d}}{\lambda}
\le C\sqrt{d}/t.
\end{aligned}
$$
with probability at least $1-t^{\alpha}$. The second last inequality uses the upper bounds on $\|\bw\|$ and $\|\bE\|$. The last inequality now follows since $\lambda>Cd^{\tfrac{1}{2}+\tfrac{1}{\alpha\mathbbm{1}(\alpha <4)}}$. Hence for any  $\delta>0,$
$$
\begin{aligned}
&	\underset{\bv:\|\bv\|=1,\,\bv\perp \bu^{(1)}}{\sup}
	\dfrac
	{\mathscr{E}_2\times_1\bv\times_2\hat{\bv}_2\dots \times_{p-1}\hat{\bv}_{p-1}\times_p\hat{\bu}_{p, J_2}}
	{
	\|\scrX_2\times_2\hat{\bv}_2\dots \times_{p-1}\hat{\bv}_{p-1}\times_p\hat{\bu}_{p, J_2}\|}\\
&\le \dfrac{C\sqrt{d}/t}{\lambda\|\bu_{p,J_2}\|(1-\eta^2)^{(p-1)/2}-C\sqrt{d}/t}\\
&\le \dfrac{C\sqrt{d}}{\lambda t}
\end{aligned}
$$
with probability at least $1-t^{\alpha}$. The proof for the rest of the modes follows similarly. Finally, initializing with $\scrX_2$ and using $\scrX_1$ for optimal estimation, we also have $\hat{\bu}_{p,J_1}$ that is a rate  optimal estimator of $\bu_{p,J_1}.$ This finishes the proof.
\end{proof}

\bibliographystyle{plainnat}
\bibliography{references}

\appendix
\section{Proof of Lemma \ref{lem:mat_means}}\label{subsec:lems}

\begin{proof}
	We write $\bw=(\bu_2\,\odot\,\bu_3\,\dots\odot\,\bu_p).$ By definition
	$$\bS_j=\dsum_{i_{-1}\in I_j}\left(\bX_{i_{-1}}\bX_{i_{-1}}^\top-\mathrm{diag}(\bX_{i_{-1}}\bX_{i_{-1}}^\top)\right).$$
	Notice that $$\EE(\bX_{i_{-1}}\bX_{i_{-1}}^\top)=\lambda^2\bw^2_{i_{-1}}\bu_1\otimes\bu_1+I,$$ which implies
	$$
	\EE(\bS_j|I)=\sum_{i_{-1}\in I_j}\lambda^2\bw^2_{i_{-1}}(\bu_1\otimes\bu_1-\mathrm{diag}(\bu_1\otimes\bu_1))=\left(\sum_{i_{-1}\in I_j}\bw^2_{i_{-1}}\right)\bV_1.
	$$
	Next, for any $i_{-1}\in I_j$ and $s,t\in[d],$ $s\neq t$
	$$
	\begin{aligned}
		&\EE\left[\left(\bX_{i_{-1}}\bX_{i_{-1}}^\top-\mathrm{diag}(\bX_{i_{-1}}\bX_{i_{-1}}^\top)\right)^2_{st}\right]\\
		=&\EE\left(\sum_{l\neq s,\,t}(\bX_{i_{-1}})_s(\bX_{i_{-1}})_t(\bX_{i_{-1}})_l^2\right)
		\\=& \EE(\bX_{i_{-1}})_s\EE(\bX_{i_{-1}})_t\sum_{l\neq s,\,t}\EE(\bX_{i_{-1}})_l^2
		\\=& \lambda^2\bw^2_{i_{-1}}(\bu_1)_s(\bu_1)_t\sum_{l\neq s,\,t}(\lambda^2\bw^2_{i_{-1}}(\bu_1)_l^2+1)
		\\=&\lambda^4\bw_{i_{-1}}^4(\bu_1)_s(\bu_1)_t(1-(\bu_1)_s^2-(\bu_1)_t^2)+\lambda^2(d-2)\bw^2_{i_{-1}}(\bu_1)_s(\bu_1)_t.
	\end{aligned}
	$$
	On the other hand,
	$$
	\begin{aligned}
		&\EE\left[\left(\bX_{i_{-1}}\bX_{i_{-1}}^\top-\mathrm{diag}(\bX_{i_{-1}}\bX_{i_{-1}}^\top)\right)^2_{ss}\right]\\
		=&\EE\left(\sum_{l\neq s}(\bX_{i_{-1}})^2_s(\bX_{i_{-1}})_l^2\right)
		\\=&(\lambda^2\bw^2_{i_{-1}}(\bu_1)_s^2+1)\sum_{l\neq s} (\lambda^2\bw^2_{i_{-1}}(\bu_1)_l^2+1)
		\\=&\lambda^4\bw^4_{i_{-1}}(\bu_1)_s^2(1-(\bu_1)_s^2)+\lambda^2\bw^2_{i_{-1}}(1+(d-2)(\bu_1)_s^2)+d-1.
	\end{aligned}
	$$
	Collecting all the terms,
	$$
	\begin{aligned}
		\EE[\left(\bX_{i_{-1}}\bX_{i_{-1}}^\top-\mathrm{diag}(\bX_{i_{-1}}\bX_{i_{-1}}^\top)\right)^2]
		=&\bw^4_{i_{-1}}\bV_{1}^{2}+\lambda^2\bw^2_{i_{-1}}(d-2)\bu_1\bu_1^\top+[(d-1)+\lambda^2\bw^2_{i_{-1}}]I_d.
	\end{aligned}
	$$
	Similarly for $i_{-1,1}\neq i_{-1,2}\in I_j,$ and indices $s,t\in[d],$
	$$
	\begin{aligned}
	\EE&\left[\left(\bX_{i_{-1,1}}\bX_{i_{-1,1}}^\top-\mathrm{diag}(\bX_{i_{-1,1}}\bX_{i_{-1,1}}^\top)\right)\left(\bX_{i_{-1,2}}\bX_{i_{-1,2}}^\top-\mathrm{diag}(\bX_{i_{-1,2}}\bX_{i_{-1,2}}^\top)\right)\right]_{st}
	\\&=\lambda^4\bw^2_{i_{-1,1}}\bw^2_{i_{-1,2}}(\bu_1)_s(\bu_1)_t(1-(\bu_1)_s^2-(\bu_1)_t^2\mathbbm{1}(s\neq t))
	\end{aligned}
	$$
	meaning 
	$$
	\EE\left[\left(\bX_{i_{-1,1}}\bX_{i_{-1,1}}^\top-\mathrm{diag}(\bX_{i_{-1,1}}\bX_{i_{-1,1}}^\top)\right)\left(\bX_{i_{-1,2}}\bX_{i_{-1,2}}^\top-\mathrm{diag}(\bX_{i_{-1,2}}\bX_{i_{-1,2}}^\top)\right)\right]=\bw^2_{i_{-1,1}}\bw^2_{i_{-1,2}}\bV_1^2.
	$$
	Adding the terms above,
	\begin{align*}
	\EE(\bS_j^2|I)=&\left(\sum_{i_{-1}\in I_j}\bw^2_{i_{-1}}\right)^2\bV_1^2+\lambda^2(d-2)\bu_1\bu_1^\top\sum_{i_{-1}\in I_j}\bw^2_{i_{-1}}+\sum_{i_{-1}\in I_j}[(d-1)+\lambda^2\bw^2_{i_{-1}}]I_d
	\\=&\left(\sum_{i_{-1}\in I_j}\bw^2_{i_{-1}}\right)^2\bV_1^2+\lambda^2\left[(d-2)\bu_1\bu_1^\top+I_d\right]\sum_{i_{-1}\in I_j}\bw^2_{i_{-1}}+(d-1)|I_j|I_d.
	\end{align*}
	Consequently, conditional on $I$,
	\begin{align*}
	\sigma^2=&\norm*{\dfrac{1}{n}\sum_{j=1}^n\EE(\bS_j^2|I)}
	\\=&\dfrac{1}{n}\norm*{\sum_{j=1}^n\left[\left(\sum_{i_{-1}\in I_j}\bw^2_{i_{-1}}\right)^2\bV_1^2+\lambda^2\left[(d-2)\bu_1\bu_1^\top+I_d\right]\sum_{i_{-1}\in I_j}\bw^2_{i_{-1}}+(d-1)|I_j|I_d\right]}
	\\\le &\, \lambda^4\left(\dfrac{1}{n}\sum_{j=1}^n\left(\sum_{i_{-1}\in I_j}\bw^2_{i_{-1}}\right)^2\right)+\dfrac{\lambda^2d}{n}\sum_{j=1}^n\sum_{i_{-1}\in I_j}\bw_{i_{-1}}^2+{d\over n}\sum_{j=1}^n|I_j|.
	\end{align*}
\end{proof}
	
\end{document}